\documentclass[12pt, reqno, a4paper]{amsart}
\usepackage[utf8]{inputenc}
\usepackage{amsmath,amssymb,amsthm, todonotes,url,mathtools}
\usepackage{hyperref,todonotes}
\usepackage[lmargin=35mm,rmargin=35mm,bmargin=30mm,tmargin=30mm]{geometry}
\usepackage[foot]{amsaddr}
\usepackage{subcaption}
\usepackage[inline]{enumitem}

\newtheorem{theorem}{Theorem}[section]

\newtheorem{lemma}[theorem]{Lemma}
\newtheorem{corollary}[theorem]{Corollary}
\newtheorem{proposition}[theorem]{Proposition}
\newtheorem{observation}[theorem]{Observation}

\newtheorem*{claim*}{Claim}
\newtheorem{problem}[theorem]{Problem}

\makeatletter
\newenvironment{claimproof}[1][\proofname]{\par\pushQED{\hfill$\lozenge$}\normalfont \topsep6\p@\@plus6\p@\relax \trivlist \item\relax{\itshape#1\@addpunct{.}}\hspace\labelsep\ignorespaces}{\popQED\endtrivlist\@endpefalse}
\makeatother

\makeatletter
\newtheorem*{rep@theorem}{\rep@title}
\newcommand{\newreptheorem}[2]{
\newenvironment{rep#1}[1]{
 \def\rep@title{#2 \ref{##1}}
 \begin{rep@theorem}}
 {\end{rep@theorem}}}
\makeatother

\newtheorem{conjecture}[theorem]{Conjecture}
\newtheorem*{conjecture*}{Conjecture}
\newreptheorem{conjecture}{Conjecture}

\theoremstyle{definition}                    

\theoremstyle{remark}   

\newtheorem*{remark*}{Remark}

\newtheorem*{assump*}{\bf{Standard notation}}

\numberwithin{equation}{section}

\newcommand{\floor}[1]{\left\lfloor #1 \right\rfloor}

\DeclareMathOperator{\Span}{span}
\DeclareMathOperator{\ord}{ord}
\newcommand{\F}{\mathbb{F}}
\newcommand{\Z}{\mathbb{Z}}
\newcommand{\N}{\mathbb{N}}
\newcommand{\red}{\mathrm{red}}
\newcommand{\green}{\mathrm{green}}
\newcommand{\blue}{\mathrm{blue}}
\usepackage{dsfont}

\DeclareMathAlphabet{\mathantt}{OT1}{antt}{m}{it}
\newcommand{\e}{\mathantt{e}} 

\newcommand{\dju}{\amalg}

\title{Cyclically covering subspaces in $\mathbb{F}_2^\MakeLowercase{n}$}
\date{}
\author{James Aaronson}
\address{James Aaronson, Mathematical Institute, University of Oxford}
\email{\{aaronson}

\author{Carla Groenland}
\address{Carla Groenland, Mathematical Institute, University of Oxford}
\email{groenland}

\author{Tom Johnston}
\address{Tom Johnston, Mathematical Institute, University of Oxford}
\email{thomas.johnston\}@maths.ox.ac.uk}

\makeatletter
\def\paragraph{\@startsection{paragraph}{4}%
  \z@\z@{-\fontdimen2\font}
  {\normalfont\bfseries}}
\makeatother

\begin{document}

\begin{abstract}
A subspace of $\mathbb{F}_2^n$ is called cyclically covering if every vector in $\mathbb{F}_2^n$ has a cyclic shift which is inside the subspace. Let $h_2(n)$ denote the largest possible codimension of a cyclically covering subspace of $\mathbb{F}_2^n$. We show that $h_2(p)= 2$ for every prime $p$ such that 2 is a primitive root modulo $p$, which, assuming Artin's conjecture, answers a question of Peter Cameron from 1991. We also prove various bounds on $h_2(ab)$ depending on $h_2(a)$ and $h_2(b)$ and extend some of our results to a more general set-up proposed by Cameron, Ellis and Raynaud.
\end{abstract}

\maketitle

\section{Introduction}
Let $q$ be a prime power. For $n \in \N$, let $\{e_0, e_1, \dots, e_{n-1} \}$ be the standard basis for $\F_q^n$. Throughout the paper, the indices of vectors in $\F_q^n$ will be taken modulo $n$ (in particular, we set $e_n = e_0$). Define the \emph{cyclic shift} operator $\sigma:\F_q^n\to \F_q^n$ by
\begin{align*}
\sigma \left(\sum_{i=0}^{n-1} x_i e_i\right) = \sum_{i=0}^{n-1} x_{i} e_{i+1}.
\end{align*}
We say that a subspace $U \leq \mathbb{F}_q^n$ is \emph{cyclically covering} if $\bigcup_{i=0}^{n-1} \sigma^i (U) = \mathbb{F}_q^n$. For any $n \in \mathbb{N}$, let $h_q(n)$ denote the largest possible codimension of a cyclically covering subspace of $\F_q^n$. 

We will be primarily interested in cyclically covering subspaces of $\mathbb{F}_2^n$ and, in particular, the following problem posed (in an equivalent form) by Peter Cameron in 1991 (see \cite[Problem 190]{Cameron94}). 
\begin{problem}
Does $h_2(n)\to \infty$ as $n\to \infty$ over the odd integers or is $h_2(n)=2$ for infinitely many odd $n$?
\end{problem}
We note that for all odd $n \geq 5$, Cameron, Ellis and Raynaud \cite{CameronEllisRaynaud18} give an explicit construction for a cyclically covering subspace with codimension 2, which establishes the lower bound $h_2(n) \geq 2$. 
The motivation for the original formulation came from proving lower bounds for Isbell's conjecture (stated in \cite{CameronFranklKantor89}), and we refer the reader to the paper by Cameron, Ellis and Raynaud \cite{CameronEllisRaynaud18} for further discussion.

Cameron, Ellis and Raynaud \cite{CameronEllisRaynaud18} show that $h_2(mn)\geq \max\{h_2(m),h_2(n)\}$ for $m, n \in \N$. It follows that there are three possibilities for the behaviour of $h_2(n)$ as $n\to \infty$ over the odd integers: \begin{enumerate}
    \item either there is a prime $p > 2$ such that $h_2(p^n)$ is bounded over all $n \in \mathbb{N}$,
    \item or there is some $M \in \mathbb{N}$ such that there are infinitely many primes $p$ with $h_2(p) \leq M$,
    \item or $h_2(n) \to \infty$ over odd integers $n$. 
\end{enumerate}
In Section \ref{sec:multbounds}, we improve the stated lower bound by showing that $h_2(mn)\geq h_2(m)+h_2(n)$ for all $m,n\in \N$, which rules out the first case. The main result of our paper shows that, provided there are infinitely many primes with 2 as a primitive root, the correct case is the second one.
\begin{theorem}\label{thm:main_theorem}
Suppose that $p$ is a prime for which 2 is a primitive root. Then $h_2(p) = 2$.
\end{theorem}

Artin conjectured that there are infinitely many primes for which 2 is a primitive root; such primes are now known as Artin primes. More generally, Artin's conjecture states that for any $n \in \N$ which is not a perfect square, there are infinitely many primes $p$ such that $n$ is a primitive root modulo $p$. Artin's conjecture is widely believed; in particular, it follows as a consequence of the generalised Riemann hypothesis, as shown by Hooley \cite{Hooley67}.  While there are no values $n$ for which Artin's conjecture is known to hold, Heath-Brown \cite{HeathBrown86} has shown that the conjecture holds for at least one value of $n$ in $\{2,3,5\}$. In fact, he showed that the conjecture can only fail to hold for at most 2 primes.

Rather than considering cyclically covering subspaces directly, we will consider the equivalent problem in the orthogonal complement. We say that a vector $v \in \mathbb{F}_2^n$ \emph{works} if for every $x \in \mathbb{F}_2^n$ there is a $k$ (which we may take to be in $\{0,\dots, n-1\}$) such that $v \cdot \sigma^k x = 0$, and that the vectors $v^{(1)}, v^{(2)}, \dotsb , v^{(m)}$ \emph{work together} if, for every $x \in \mathbb{F}_2^n$, there is a $k \in \{0, \dots, n-1\}$ such that \[ v^{(1)} \cdot \sigma^k x = v^{(2)} \cdot \sigma^k x = \dotsb = v^{(m)} \cdot \sigma^k x = 0 .\] 

Suppose $U$ is a cyclically covering subspace of $\mathbb{F}_2^n$ and let $v^{(1)}, \dots, v^{(m)}$ be a basis for $U^\perp$, the orthogonal complement of $U$. Since $U$ is cyclically covering, for any $x \in \mathbb{F}_2^n$, there exists $u \in U$ and a $k$ such that $x = \sigma^k u$, so, by definition, \[  v^{(i)} \cdot \sigma^{n-k} x = v^{(i)} \cdot u  = 0 \quad \forall i\in[m]\] and the vectors $\{ v^{(1)}, \dots, v^{(m)}\}$ work together. 

Conversely, if we have a set $ V = \{ v^{(1)}, \dots, v^{(m)}\}$ of vectors which work together, then $U = \Span(V)^{\perp}$ is cyclically covering. Indeed, if $x \in \mathbb{F}_2^n$, then by definition there is a $k$ such that \[ v^{(1)} \cdot \sigma^k x = v^{(2)} \cdot \sigma^k x = \dotsb = v^{(m)} \cdot \sigma^k x = 0. \] Hence, $\sigma^k x \in U$, and $x \in \sigma^{n-k} U$, which shows $U$ is cyclically covering. 

This means that $h_2(n)$ is the largest value $m \in \Z$ such that there exist $v^{(1)},\dots,v^{(m)} \in \F_2^n$ which work together and are linearly independent.

In Section \ref{sec:multbounds} we use this formulation to prove the lower bound $h_2(mn) \geq h_2(m) + h_2(n)$, and to give an upper bound of $h_2(2n) \leq 2h_2(n)$. In this formulation, there is a direct correspondence between vectors that work and cyclically covering subspaces with codimension 1, and we give a precise characterisation of these vectors in Theorem \ref{thm:elts_that_work}.

Sections \ref{sec:rootsandconj}, \ref{sec:progress} and \ref{sec:proof} are all devoted to the proof of our main result, Theorem \ref{thm:main_theorem}.  Suppose towards a contradiction that we could find three linearly independent vectors that work together for $p$ a prime with 2 as a primitive root. In Lemma \ref{lem:irreducible} we use the fact that $\F_2[X]/\langle 1+X+\dots+X^{p-1}\rangle$ is a field to show that we can then also find such a triple for which one of the vectors is $\e=(0,1,\dots,1)$. In the remainder of the proof, we show for general primes $p$ that there are no linear independent $\e,v,w\in \F_2^p$ that work together. 

If $\e,v$ and $w$ work together, then in particular $\e$ and $v$ work together. Based on computer experiments, we have a very precise conjecture of which vectors $v$ work with $\e$.
We call a vector $v \in \F_2^n$ \emph{symmetric} if it is of the form $(v_0,v_1,v_2,\dots,v_2,v_1)$ (in other words, $v_i = v_j$ whenever $i \equiv - j \mod n$). 
\begin{conjecture}\label{conj:our_in_intro}
Suppose that $n\in \N$ is an odd number that is not divisible by 7. Let $v \in \F_2^n$, and suppose that $v$ and $\e$ work together. Then $v$ is symmetric.
\end{conjecture}
In Lemma \ref{lem:nothing_works_with_sym} we prove that if $\e$, $v$ and $w$ work together and are linearly independent, neither $v$ nor $w$ can be symmetric. Thus, Conjecture \ref{conj:our_in_intro} would imply Theorem \ref{thm:main_theorem}. However, we unfortunately were unable to prove this conjecture.

In Section \ref{sec:progress}, we make partial progress towards the conjecture. To each vector $v\in \F_2^p$, we associate the Cayley digraph $G_v$ on vertex set $\Z/p\Z$ with the set $A$ of indices $i$ with $v_i=1$ as generating set (that is, $(i,j)$ is an arc if and only if $v_{j-i}=1$). This graph is simple\footnote{We say a digraph $(V,E)$ is simple if it has no self-loops and $(i,j)\in E$ if and only if $(j,i)\in E$.} if and only if $v$ is symmetric.
We show in Proposition \ref{prop:graph_theory_formulation} that $\e$ and $v$ work together if and only if $G_v$ has no induced subgraph on an odd number of vertices where each vertex has odd outdegree. This shows Conjecture \ref{conj:our_in_intro} is equivalent to the following conjecture.
\begin{conjecture}\label{conj:second_formulation}
Suppose that $n\in \N$ is an odd number that is not divisible by 7. Let $G$ be a Cayley digraph on $\Z/n\Z$. Then $G$ is a simple graph if and only if $G$ has no odd-sized induced subgraph where each vertex has odd outdegree. 
\end{conjecture}
Note that one of the directions follows from the Handshaking Lemma. For non-symmetric $v$, the goal is now to find a `bad' subgraph in the graph $G_v$ in order to show $\e$ and $v$ do not work together. We define the \emph{girth} of $v$ to be the length of the shortest directed cycle in $G_v$. In Proposition \ref{prop:girth_is_small}, we reduce to the case in which the girth of $v$ is $4$ or $6$. For this, we use several additive results. The main observation is that the generating set $A\subseteq\Z/p\Z$ has nice additive properties if $G_v$ has no small cycles. This then allows us to conclude that $A$ is contained in a small arithmetic progression, which gives us enough control on the edges of $G_v$ to find the `bad' subgraph.

Even though we cannot prove Conjecture \ref{conj:our_in_intro}, we make enough progress in Section \ref{sec:progress} to push through the proof of Theorem \ref{thm:main_theorem} in Section \ref{sec:proof} via a case analysis on the girth of the vectors $v,w$ and $v+w$.

\smallskip

In Section \ref{sec:hqn}, we consider the problem of bounding $h_q(n)$. This was studied by Cameron, Ellis and Raynaud \cite{CameronEllisRaynaud18} and, amongst many other results, they show that the lower bound $h_2(mn) \geq \max\{h_2(m), h_2(n)\}$ mentioned earlier holds even when $2$ is replaced by any prime power $q$. We show how our lower bound also holds in $\F_q^n$ and extend the upper bound given in Theorem \ref{thm:mult2} to this setting as well. Finally, we show that $h_q(p) = 0$ when $q$ is a prime and $p > q$ is a prime with $q$ as a primitive root.

Our conjecture, and many other interesting questions outlined in Section \ref{sec:concl}, are left open.

\subsection{Notation}
We use the notation $[n]=\{1,\dots,n\}$. For a vector $x \in \mathbb{F}_q^n$, we denote the $i$th coordinate by $x_i \in \mathbb{F}_q$ and we write $|x|=\sum_{i=1}^n x_i \in \mathbb{F}_q$ for the sum of the coordinates (over $\mathbb{F}_q$). For vectors $v\in \F_q^m$ and $w\in \F_q^n$, we will write $(v,w)\in \F_q^{m+n}$ for their concatenation.

For two sets $A$ and $B$, we use the standard notation $A+B=\{a+b:a\in A,b\in B\}$ and use $nA$ for the sum of $n$ copies of $A$. 

\section{Multiplicative bounds}
\label{sec:multbounds}
Cameron, Ellis and Raynaud \cite[Lemma 3]{CameronEllisRaynaud18} proved that for any $m, n \in \N$, $h_2(mn)\geq \max\{h_2(n),h_2(m)\}$. We offer the following improvement to this lower bound.
\begin{theorem}\label{thm:sum_mult_bound}
For any $m,n \in \N$, 
\[ h_2(mn) \geq h_2(m) + h_2(n) \]
\end{theorem}
Theorem \ref{thm:sum_mult_bound} immediately tells us that $h_2(n_i) \rightarrow \infty$ for many sequences $n_i$.
\begin{corollary}
$h_2(m^n) \to \infty$ as $n \to \infty$ for any $m$ for which $h_2(m)>0$.
\end{corollary}
Indeed, Theorem \ref{thm:sum_mult_bound} tells us that $h_2(m^n) \geq nh_2(m)$. Note that the condition $h_2(m)>0$ is not very restrictive: Cameron, Ellis and Raynaud \cite{CameronEllisRaynaud18} proved that $h_2(m)=0$ if and only if $m$ is a power of 2. We will provide an alternative proof of this fact later in this section.

We remark that in the proof Theorem \ref{thm:sum_mult_bound} below, we may replace $2$ by any prime $q$ to give the more general variant stated in Theorem \ref{thm:sum_mult_bound_q}.
\begin{proof}[Proof of Theorem \ref{thm:sum_mult_bound}]
Let $v^{(1)}, \dots, v^{(a)}\in \F_2^m$ be a set of $a = h_2(m)$ linearly independent vectors which work together for $m$, and $w^{(1)}, \dots, w^{(b)} \in\F_2^n$ a set of $b = h_2(n)$ linearly independent vectors which work together for $n$. To prove the theorem we now give a family of $a + b$ linearly independent vectors in $\F_2^{mn}$ which work together for $mn$.

Firstly, for each $i \in [a]$ and $j\in [b]$, define the vectors $\widetilde{v}^{(i)} \in \mathbb{F}_2^{mn}$ and $\widetilde{w}^{(j)} \in \mathbb{F}_2^{mn}$ as 
\begin{align*}
\widetilde{v}^{(i)}&= (\underbrace{v^{(i)}_0,\dots, v^{(i)}_0}_{n \text{ copies}},\dots, \underbrace{v^{(i)}_{m-1},\dots, v^{(i)}_{m-1}}_{n \text{ copies}}),\\
\widetilde{w}^{(j)}&= (\underbrace{w^{(j)},\dots, w^{(j)}}_{m \text{ copies}}).\\
\end{align*}
We claim that the family of vectors $\widetilde{v}^{(1)}, \dots, \widetilde{v}^{(a)}, \widetilde{w}^{(1)}, \dots, \widetilde{w}^{(b)}$ all work together and are linearly independent, which will imply that $h_2(mn) \geq a + b$.

First, we prove that the vectors work together; in other words, we will prove that for any $x \in \F_2^{mn}$, there is a $k\in \Z$ such that $\widetilde{v}^{(i)} \cdot \sigma^k x = \widetilde{w}^{(j)} \cdot \sigma^k x = 0$ for all $i\in [a]$ and $j\in [b]$. 

Suppose $\widetilde{x} \in \mathbb{F}_2^{mn}$. Let $y\in \F_2^n$ be the vector defined by 
\[
y=\left(\sum_{r=0}^{m-1} \widetilde{x}_{rn},\sum_{r=0}^{m-1} \widetilde{x}_{rn+1}, \dots, \sum_{r=0}^{m-1} \widetilde{x}_{rn+n-1}\right).
\]
This means that $\widetilde{w}^{(j)} \cdot \sigma^{k} \widetilde{x} = w^{(j)} \cdot \sigma^{k} y$ for all $k\in \Z$. Since the $w^{(j)}$ work together in $\F_2^n$, there is some choice of $k \in \Z$ such that $w^{(j)} \cdot \sigma^{k}y = 0$ for each $j$, which means that $\widetilde{w}^{(j)} \cdot \sigma^{k} \widetilde{x} = 0$ for each $j$. We may assume that $k=0$ by replacing $\widetilde{x}$ with $\sigma^k\widetilde{x}$.

Let $z \in \F_2^m$ be given by 
\[
z = \left(\widetilde{x}_0+\dots+\widetilde{x}_{n-1},\widetilde{x}_n+\dots+\widetilde{x}_{2n-1},\dots,\widetilde{x}_{mn-n}+\dots+\widetilde{x}_{mn-1} \right).
\]
Then
\[
\widetilde{v}^{(i)} \cdot \sigma^{n\ell} \widetilde{x} = v^{(i)} \cdot \sigma^{\ell} z \text{ for all }i\in[a] \text{ and }\ell \in\Z.
\]
Since $v^{(1)}, \dots, v^{(a)}$ work together for $m$, there is some choice of $\ell \in[m]$ such that $v^{(i)} \cdot \sigma^\ell  z = 0$ for every $i \in [a]$. We conclude
\begin{align*}
     \widetilde{v}^{(i)} \cdot \sigma^{\ell n} \widetilde{x} &= v^{(i)} \cdot \sigma^\ell z = 0 \quad \forall i \in [a],\\
    \widetilde{w}^{(j)} \cdot \sigma^{\ell n} \widetilde{x} &= w^{(j)} \cdot  y = 0\quad \forall j \in [b].
\end{align*}
This shows the vectors $  \widetilde{v}^{(1)}, \dots,   \widetilde{v}^{(a)}, \widetilde{w}^{(1)}, \dots, \widetilde{w}^{(b)}$ all work together, and it only remains to show that the vectors are linearly independent in $\F_2^{mn}$. 

First, observe that the vectors $\widetilde{v}^{(i)} \in \F_2^{mn}$ for $i \in [a]$ are linearly independent because the vectors $v^{(i)} \in \F_2^m$ are linearly independent. Similarly, the family $\left\{\widetilde{w}^{(j)} | j \in [b] \right\}$ is linearly independent because the family $\left\{w^{(j)}| j \in [b]\right\}$ is.

To show that the vectors $\widetilde{v}^{(1)}, \dots,   \widetilde{v}^{(a)}, \widetilde{w}^{(1)}, \dots, \widetilde{w}^{(b)}$ are linearly independent, it suffices to show that any vector that is in the span of the $\widetilde{v}^{(i)}$ and also in the span of the $\widetilde{w}^{(j)}$ is the zero vector. Suppose that $\widetilde{u}$ is such a vector.

By considering the standard basis vector $x=e_0$ in the definition of working together for $v^{(1)}, \dots, v^{(a)}$, it follows that there must be an $r$ such that $v^{(i)}_{r} = 0$ for all $i \in [a]$. In particular, $\widetilde{u}_{rn+s} = 0$ for each $s \in \{0,1,\dots,n-1\}$.

However, the fact that the $w^{(j)}$ are linearly independent means that the only linear combination $\widetilde{w}$ of the $\widetilde{w}^{(j)}$ for which $\widetilde{w}_{rn+s} = 0$ for each $s \in \{0,1,\dots,n-1\}$ is zero. Thus, $\widetilde{u} = 0$, which means that the vectors $\widetilde{v}^{(1)}, \dots, \widetilde{v}^{(a)}, \widetilde{w}^{(1)}, \dots, \widetilde{w}^{(b)}$ are linearly independent, as required.
\end{proof}
We also obtain the following rough upper bound.
\begin{theorem}
\label{thm:mult2}
For any $n \geq 0$, 
\begin{equation}\label{eqn:mult2}
    h_2(2n) \leq 2h_2(n).
\end{equation}
\end{theorem}
\begin{proof}
Suppose $\{v^{(1)}, \dots, v^{(a)}\}\subseteq \mathbb{F}_2^{2n}$ is a collection of $a = h_2(2n)$ linearly independent vectors which work together for $2n$ and write $v^{(i)} = (u^{(i)} + w^{(i)}, w^{(i)})$ for $u^{(i)}, w^{(i)} \in \F_2^n$. The vectors $u^{(1)}, \dots, u^{(a)}\in \F_2^n$ work together. To see this, note that for any $y \in \F_2^n$, there is a $k$ such that \[ u^{(i)} \cdot \sigma^k y = v^{(i)} \cdot \sigma^k (y,y) = 0\] for every $i$.

Without loss of generality $u^{(1)}, \dots, u^{(\ell)}$ is a maximal linearly independent subset of $\{u^{(1)},\dots,u^{(a)}\}$. Then for $j > \ell$ we can find $\lambda_1, \dots, \lambda_\ell$ such that $u^{(j)} = \lambda_1 u^{(1)} + \dotsb + \lambda_\ell u^{(\ell)}$. We may replace $v^{(j)}$ with $\lambda_1 v^{(1)} + \dotsb +\lambda_\ell v^{(\ell)} + v^{(j)}$ without changing the span of the $v^{(i)}$ or the fact that they work together. This will give the vector \[ \left( \lambda_1 w^{(1)} + \dotsb + \lambda_\ell w^{(\ell)} + w^{(j)}, \lambda_1 w^{(1)} + \dotsb + \lambda_\ell w^{(\ell)} + w^{(j)} \right).\]
Doing this for all $j > \ell$, we can assume that $u^{(1)}, \dots, u^{(\ell)}$ are linearly independent (hence $\ell\leq h_2(n)$) and that $u^{(i)}=0$ for $i>\ell$. 

In particular, for each $i > \ell$, we have $v^{(i)} = (w^{(i)}, w^{(i)})$. We claim that the vectors $w^{(i)}$ with $i > \ell$ are linearly independent and work together. This implies that $h_2(2n) - h_2(n) \leq h_2(n)$, which gives the desired result.

The linear independence of the $w^{(i)}$ follows from the linear independence of the $v^{(i)}$. To show that the $w^{(i)}$ work together we must prove that, for any $y \in \F_2^n$, there is some $k$ such that $w^{(i)} \cdot \sigma^k y = 0$ for all $i > \ell$.

Since the $v^{(i)}$ work together, there must exist a $k < 2n$ such that $v^{(i)} \cdot \sigma^k(0,y) = 0$ for each $i > \ell$. But we also have that $v^{(i)} = (w^{(i)}, w^{(i)})$, so
\[ v^{(i)} \cdot \sigma^k y = (w^{(i)},w^{(i)}) \cdot \sigma^k(0,y) = w^{(i)} \cdot \sigma^k y.\]
Thus, for this choice of $k$, we have that $w^{(i)} \cdot \sigma^k y = 0$ for all $i > \ell$, as required.
\end{proof}

It is unlikely that \eqref{eqn:mult2} is tight in general; indeed, if $n$ is such that $h_2(n) > \frac{1}{2} \log_2(n) + 1$, then the tightness of \eqref{eqn:mult2} would contradict the upper bound $h_2(2n) \leq \log_2(n) + 1$. However, we may combine Theorem \ref{thm:mult2} with the fact that $h_2(1) = 0$ to deduce the following.
\begin{corollary}
\label{cor:powersoftwo}
$h_2(2^i)=0$ for every $i\in \N$.
\end{corollary}

Let $W(n)$ be the set of $v\in \F_2^n$ which work. The fact that $h_2(2^i)=0$ implies that no non-zero vector can work whenever $n$ is a power of 2 (the zero vector always works). We determine the exact structure of $W(n)$ whenever $n$ is not a power of 2, which in particular implies $h_2(n)\geq 1$ when $n$ is not a power of 2.
\begin{theorem}
\label{thm:elts_that_work}
For $n$ odd, the vectors that work are given by
\[
W(n)=\{v\in \F_2^n:|v| =  0\}.
\]
For $a$ odd and $b\in \N$,
\[
W(a2^b)= \left\{\left(v_0^{(1)},\dots ,v_0^{(2^b)},v_1^{(1)},\dots ,v_1^{(2^b)},\dots ,v_{a-1}^{(1)},\dots ,v_{a-1}^{(2^b)}\right):v^{(1)},\dots,v^{(2^b)}\in W(a)\right\}.
\]
\end{theorem}
\begin{remark*}
Since $W(1) = \{0\}$, this shows $W(2^b) = \{0\}$ for all $b \geq 0$, which agrees with Corollary \ref{cor:powersoftwo}.
\end{remark*}
\begin{proof}
Suppose first that $v \in \F_2^n$ (where $n$ is odd) is a vector that works. There must be some $k$ such that $v \cdot \sigma^k (1,\dots,1) = 0$. However, for every shift $k\in[n]$,
\[
v\cdot \sigma^k (1,\dots,1)=|v|.
\]
Hence any vector which works must satisfy $|v| =  0$. 

Conversely, suppose that $|v| = 0$. Given $x \in \F_2^n$, the number of shifts $k$ with $v \cdot \sigma^k x = 1$ is given by
\[
|\{k:v\cdot \sigma^k x = 1 \mod 2\}| = v\cdot \sum_{k=0}^{n-1}\sigma^k x = |v||x|.
\] 
Since $|v| = 0$, this shows the number of shifts $k$ with $v \cdot \sigma^k x = 1$ is even. Since $n$ is odd,  there are an odd number of shifts $k$ with $v \cdot \sigma^k x = 0$. In particular, there is at least one such $k$, proving the first claim.

Suppose now that $n$ is of the form $a2^b$ where $a$ is odd. Let $v\in W(n)$ be given and split it into components $v^{(1)}, \dots, v^{(2^b)}$ by writing
\begin{equation}
\label{eq:vcomps}
v= (v_0^{(1)},\dots, v_0^{(2^b)},v_1^{(1)},\dots, v_1^{(2^b)},\dots, v_{a-1}^{(1)},\dots, v_{a-1}^{(2^b)}).
\end{equation}

Let $w\in \F_2^{2^b}$ be the vector with coordinates $w_j=|v^{(j)}|$. Suppose that $w\neq 0$.
By Corollary \ref{cor:powersoftwo}, we know that there is some vector $z\in \F_2^{2^b}$ for which all shifts fail, or in other words, such that
\[
w \cdot \sigma^k z = 1 ~~\forall k\in[2^b].
\]
Now consider the vector $x\in\F_2^n$ given by 
\[ x = (\underbrace{z, z, \dots,  z}_{\text{$a$ times}}) = \left( z_1, \dots, z_{2^b}, z_1, \dots, z_{2^b}, \dots, z_1, \dots z_{2^b}\right).\]
Then, noting that $a$ is odd,
\[
v \cdot \sigma^k x  = a(w \cdot \sigma^{k} z)  = 1 \quad \forall k\in[n].
\]
This contradicts the assumption that $v$ works. Hence we find that $w_j=|v^{(j)}|\equiv 0$ for all $j\in [2^b]$, which shows that $v^{(1)},\dots,v^{(2^b)}\in W(a)$.

Conversely, suppose $v$ is given in the form (\ref{eq:vcomps}) for $v^{(1)},\dots,v^{(2^b)}\in W(a)$. Let $x\in \F_2^n$ be arbitrary, and split it into components $x^{(1)},\dots,x^{(2^b)}$ in the same way. We will show that
\[
|\{k\in[a]: v \cdot \sigma^{2^b k}x = 0 \}|
\]
is odd.
Note that 
\begin{align*}
\sum_{k=1}^a v \cdot \sigma^{2^b k}x&= \sum_{k=1}^{a} \left(\sum_{j=1}^{2^b} v^{(j)}\cdot \sigma^{k} x^{(j)}\right)\\
&= \sum_{j=1}^{2^b} \left(\sum_{k=1}^a v^{(j)}\cdot \sigma^{k} x^{(j)}\right)\\
&= \sum_{j=1}^{2^b} |v^{(j)}||x^{(j)}| = 0.
\end{align*}
This shows that the number of $k\in[a]$ for which $v \cdot \sigma^{2^b k}x = 1$ has to be even. Since $a$ is odd, this proves our claim. 
\end{proof}

\section{Primitive roots and a structural conjecture}
\label{sec:rootsandconj}
In order to prove Theorem \ref{thm:main_theorem}, we must first think about the structure of $\F_2^n$. For this, it will be easiest to think about vectors as polynomials, as in \cite{CameronEllisRaynaud18}.

Given a vector $v \in \F_2^n$, let its corresponding polynomial be 
\[
f_v(X) = v_0 + v_1X + v_2X^2 + \dots + v_{n-1}X^{n-1} \in \mathbb{F}_2[X]/\langle X^n - 1 \rangle.
\]
We say a set of polynomials $f_{v^{(1)}}, \dots, f_{v^{(m)}}$ \textit{work together} if and only if, for every vector $x \in \mathbb{F}_2^n$, there is some $k$ such that the coefficient of $X^k$ in $f_{v^{(i)}}f_x$ is $0$ for each $i \in [m]$. The following proposition shows that the two definitions of working together are equivalent, and so we can use them interchangeably. 

\begin{proposition}\label{prop:think_about_polys}
The vectors $v^{(1)}, \dots, v^{(m)}$ work together if and only if the corresponding polynomials $f_{v^{(1)}}, \dots, f_{v^{(m)}}$ work together. \end{proposition}

\begin{proof}
By definition, $v^{(1)}, \dots, v^{(m)}$ work together if and only if, for every $x \in \mathbb{F}_2^n$, there is a $k$ such that $v^{(i)} \cdot \sigma^{k} x = 0$ for each $i \in [m]$.

Let $\text{rev}(x)$ be the reverse of $x$ (so $\text{rev}(x)_j = x_{-j}$ for each $j$). The constant coefficient of
\[
\left(v_0+v_1X+\dots+v_{n-1}X^{n-1}\right)
\left(x_0+x_1X+\dots+x_{n-1}X^{n-1}\right)\mod (X^{n}-1)
\]
is given by
\[
v_0x_0+v_1x_{n-1}+v_2x_{n-2}+\dots +v_{n-1}x_1=v\cdot \text{rev}(x).
\]
Similarly,
\[
v^{(i)} \cdot \sigma^k (\text{rev}(x)) = \sum_{j=0}^{n-1} v_j^{(i)} \text{rev}(x)_{j-k} = \sum_{j=0}^{n-1} v_j^{(i)} x_{k-j}
\] 
is the coefficient of $X^{k}$ in $f_{v^{(i)}} f_{x}$.

Thus, $v^{(i)} \cdot \sigma^{k} (\text{rev}(x)) = 0$ for every $i \in [m]$ if and only if the coefficient of $X^{k}$ of $f_{v^{(i)}}f_{x}$ is 0 for every $i \in [m]$. 
\end{proof}
Since $\alpha f_v+f_w=f_{\alpha v+w}$ for all $\alpha\in \F_2$ and $v,w\in \F_2^n$, a collection of vectors $v^{(1)},\dots,v^{(m)}$ is linearly independent if and only if the corresponding polynomials $f_{v^{(1)}},\dots,f_{v^{(m)}}$ are linearly independent.

\begin{remark*}
The $m = 1$ case of Proposition \ref{prop:think_about_polys} gives an alternative approach to prove Theorem \ref{thm:elts_that_work}. Indeed, Proposition \ref{prop:think_about_polys} tells us that the vector $v$ will work if and only if $f_v$ works. By definition, $f_v$ works if and only if, for every $x \in \F_2^n$, there is a $k$ with $k\in\{0,\dots,n-1\}$ such that the $X^k$ coefficient of $f_vf_x$ is 0; in other words, $f_v$ works if and only if $f_v$ does not divide $1 + X + \dots + X^{n-1}$ in $\F_2[X]/\langle X^{n} - 1\rangle$.

Since $(1+X)(1 + X + \dots + X^{n-1}) = X^n + 1$, note that the polynomial $f_v$ does not divide $1 + X + \dots + X^{n-1}$ exactly when $(1 + X)^r$ divides $f_v$, where $(1 + X)^r$ is the highest power of $1+X$ which divides $X^n + 1$. To see why, factor $X^n + 1$ as $(1+X)^r g(X)$ for some polynomial $g$ which is coprime to $(1 + X)$, and so $1 + X + \dots + X^{n-1} = (1+X)^{r-1} g(X)$. The polynomial $f_v$ divides $1 + X + \dots + X^{n-1}$ modulo $X^n + 1$ if and only if it divides $1 + X + \dots + X^{n-1}$ modulo $(1+X)^r$ and $g(X)$. Since $1 + X + \dots + X^{n-1}$ is 0 mod $g(X)$, this is equivalent to $f_v$ dividing $(1 + X)^{r-1}$ modulo $(1 + X)^r$, which is the same as $(1 + X)^{r}$ not dividing $f_v$.

Now, $r = 2^b$, where $2^b$ is the highest power of 2 which divides $n$. Thus, $f_v$ does not divide $1 + X + \dots + X^{n-1}$ exactly when $(1 + X)^{2^b}$ divides $f_v$, which is equivalent to the description given in Theorem \ref{thm:elts_that_work}.
\end{remark*}

This alternative formulation in terms of polynomials is particularly useful when $n = p$ is a prime with 2 as a primitive root. This is because of the following well-known results. 
\begin{lemma}[{\cite[Theorem 2.47 (ii)]{FiniteFieldsBook}}]\label{lem:factors_of_cyclo}
Let $p$ be an odd prime and let $\ord_p(2)$ be the multiplicative order of $2 \bmod p$. Then the polynomial $1 + X + \dots + X^{p-1}$ splits as a product of $t$ distinct irreducible factors in $\F_2[X]$, where $t = \frac{p-1}{\ord_p(2)}$.

In particular, if $2$ is a primitive root mod $p$, then $1 + X + \dots + X^{p-1}$ is irreducible in $\F_2[X]$.
\end{lemma}
\begin{lemma}[{\cite[Theorem 1.61]{FiniteFieldsBook}}]\label{lem:irreducible}
The quotient ring $\mathbb{F}_2[X]/\langle f\rangle$ is a field if and only if the polynomial $f$ is irreducible over $\mathbb{F}_2$.
\end{lemma}

The following lemma is key to our approach to proving Theorem \ref{thm:main_theorem}, and is the only part of the proof which requires 2 to be a primitive root modulo $p$.
\begin{lemma}\label{lem:wlog_0all1}
Let $p$ be a prime with $2$ as a primitive root, and define the vector $\e$ to be the vector such that $\e_0 = 0$ and $\e_i = 1$ for $i \neq 0$. Suppose that the vectors $v^{(1)}, \dots, v^{(m)}$ work together and are linearly independent. Then there is a collection of $m$ linearly independent vectors which work together, and which contains $\e$.
\end{lemma}

\begin{proof}
Each polynomial $f_{v^{(i)}}$ works individually. By Theorem \ref{thm:elts_that_work}, this means $|v^{(i)}| = 0$, which means that $f_{v^{(i)}}$ is divisible by $1+X$ as $f_{v^{(i)}}(1) = 0$. Since none of the polynomials can be the zero polynomial, they must all be equivalent to a nonzero polynomial modulo $1 + X + \dots + X^{p-1}$.

By Lemma \ref{lem:factors_of_cyclo}, $X^p - 1$ factors into irreducibles as 
\[
X^p - 1 = (X+1)(1 + X + \dots + X^{p-1}).
\]

Thus, $\F_2[X]/(X^p - 1) \cong \F_2[X]/(X+1) \oplus \F_2[X]/(1 + X + \dots + X^{p-1})$, and Lemma \ref{lem:irreducible} tells us that the two factors are fields. In particular, the irreducibility of $1 + X + \dots + X^{p-1}$ tells us that each $f_{v^{(i)}}$ is invertible modulo $1 + X + \dots + X^{p-1}$.

Thus, we may use the Chinese Remainder Theorem to find a polynomial $q$ which is $1 \bmod (1+X)$ and an inverse to $f_{v^{(1)}}$ modulo $(1 + X + \dots + X^{p-1})$. 

Now, $qf_{v^{(1)}}$ is $0 \bmod (X+1)$ and $1 \bmod (1 + X + \dots + X^{p-1})$, so
\[
qf_{v^{(1)}}=X+\dots+X^{p-1}=f_{\e}.
\]

We claim that the polynomials $q f_{v^{(1)}}, \dots, q f_{v^{(m)}}$ still work together and are linearly independent. Taking the vectors in $\F_2^p$ corresponding to the given polynomials then gives the collection of vectors required by the statement of the lemma. 

Suppose first that $q f_{v^{(1)}}, \dots, q f_{v^{(m)}}$ do not work together; thus, there exists a vector $x \in \F_2^p$ such that, for any $k$, there is an $i \in [m]$ such that the $X^k$ coefficient of $f_{x} qf_{v^{(i)}}$ is 1. Writing $f_xq=f_{x'}$ for some $x'\in \F_2^p$, we find that, for any $k$, there is an $i \in [m]$ such that the $X^k$ coefficient of $f_{x'}f_{v^{(i)}}$ is 1. This contradicts the assumption that the $v^{(i)}$ work together.

Next, suppose that the $qf_{v^{(i)}}$ are not linearly independent. Then there is a relation
\[ \sum_i\lambda_i q f_{v^{(i)}}\equiv 0\bmod (X^p-1) \] for some choice of $\lambda_i\in \F_2$ not all 0. However, $q$ is invertible modulo $X^p - 1$ by construction, so \[ \sum_i\lambda_i f_{v^{(i)}}\equiv 0\bmod (X^p-1). \] This contradicts the linear independence of the $f_{v^{(i)}}$.
\end{proof}

Lemma \ref{lem:wlog_0all1} tells us that in order to show that $h_2(p) = 2$ for a prime $p$ with 2 as a primitive root, it suffices to show that there is no collection of three linearly independent vectors $\e, v, w$ which work together. 

We now turn to examining the structure of the set of vectors $v \in \F_2^n$ which work with $\e$. There is a fairly substantial collection of vectors $v$ which work with $\e$, as shown by the following proposition.
\begin{proposition}\label{prop:all_sym_work}
Let $n$ be any odd number, and suppose $v \in \F_2^n$ is a vector with $v_0 = 0$. Suppose further that $v$ is \emph{symmetric}, in that $v_i = v_{-i}$ for each $i$. Then $v$ works together with $\e$.
\end{proposition}

\begin{proof}
Let $x \in \F_2^n$ be any vector. We must prove that there exists a $k$ such that $\e \cdot \sigma^k{x} = v \cdot \sigma^k{x} = 0$. 

First, observe that we may assume that $x$ has odd weight. If $x$ has even weight, we can add the vector $ z = (1, \dots, 1)$ of all ones; this gives a vector of odd weight. Now, $v \cdot \sigma^k{z} = \e \cdot \sigma^k{z} = 0$; to see why, observe that $v$ must have even weight by virtue of being symmetric and $\e$ must have even weight as $n$ is odd. Therefore, adding $z$ does not change the inner products $v \cdot \sigma^k x$ and $\e \cdot \sigma^k x$.

Now, $\e \cdot \sigma^k{x} = 0$ exactly when $x_{-k} = 1$, and for such a choice of $k$, 
\begin{equation}\label{eqn:v_dot_wsh}
v \cdot \sigma^k{x} = \sum_{i : x_{i-k} = 1} v_i.
\end{equation}
Summing (\ref{eqn:v_dot_wsh}) over all of the values of $k$ with $x_{-k} = 1$, we obtain
\begin{align*}
    \sum_{k : x_{-k} = 1} v \cdot \sigma^k{x} &= 
    \sum_{(k,i) : x_{-k}= x_{i-k} = 1} v_i \\
    &= \sum_{(j_1,j_2):x_{j_1} = x_{j_2} = 1} v_{j_2 - j_1}.
\end{align*}
Since $v_0 = 0$, and $v_{j_2 - j_1} + v_{j_1 - j_2} = 0$, this sum is 0.

Therefore, the sum of $v \cdot \sigma^k{x}$ over all $k$ such that $\e \cdot \sigma^k{x} = 0$ is zero. Since there are an odd number of such $k$ (as $x$ has odd weight), at least one of these $k$ must have $\e \cdot \sigma^k x = v \cdot \sigma^k{x} = 0$.
\end{proof}

Since $\e$ and $v$ work together whenever $v$ is symmetric, the following lemma will be necessary if we are to prove that $h_2(p) = 2$.
\begin{lemma}\label{lem:nothing_works_with_sym}
Let $p > 3$ be a prime, and suppose that $v \neq \e$ is a non-zero symmetric vector. Suppose further that $w \in \F_2^p$ is any vector such that $\e, v$ and $w$ work together. Then $w$ is contained in the subspace $\langle \e, v\rangle$ spanned by $\e$ and $v$.
\end{lemma}

In order to prove this, we will rely on the Cauchy-Davenport inequality \cite[Theorem 5.4]{TaoVu} and Vosper's theorem \cite[Theorem 5.9]{TaoVu}, both of which we state here for convenience.
\begin{theorem}[Cauchy-Davenport Inequality]\label{thm:cauchy-davenport}
Suppose that $p$ is a prime, and $A$ and $B$ are two nonempty sets in $\Z/p\Z$. Then
\[
|A+B| \geq \min(|A| + |B| - 1, p).
\]
\end{theorem}
\begin{theorem}[Vosper's Theorem]\label{thm:vosper}
Suppose that $p$ is a prime, and $A$ and $B$ are two sets in $\Z/p\Z$. Suppose that $|A|, |B| \geq 2$ and $|A + B| \leq p-2$. Then $|A + B| = |A| + |B| - 1$ if and only if $A$ and $B$ are arithmetic progressions with the same common difference.
\end{theorem}

We will often implicitly use the following observation, which gives a simple restriction on vectors $v$ that can work together with $\e$.
\begin{observation}\label{obs:0_in_slot_0}
Suppose that $v \in \F_2^n$ works with $\e$. Then $v_0 = 0$.
\end{observation}
This follows by considering the test vector $x = e_0$ of weight 1. Since $\sigma^kx \cdot \e$ is only zero if $k = 0$ and $\e$ works with $v$, we must have that $v_0=e_0 \cdot v = 0$.

\begin{proof}[Proof of Lemma \ref{lem:nothing_works_with_sym}]
Suppose that $v \in \F_2^p$ is symmetric, nonzero and not equal to $\e$. Let $A \subseteq \Z/p\Z$ consist of the values of $i$ for which $v_i = 1$, and let $B \subseteq (\Z/p\Z) \setminus \{0\}$ consist of the values of $i$ (except 0) for which $v_i = 0$. Observe that both $A = -A$ and $B = -B$ since $v$ is symmetric, and that $|A|,~|B| \geq 2$ as we assume that $v$ is non-zero and $v \neq \e$.

The space $\langle\e, v\rangle$ consists of the four vectors $u$ which satisfy both that $u_0 = 0$ and that $u$ is constant on both $A$ and $B$. Thus, if $w$ is a vector that works together with both $\e$ and $v$, it suffices to prove that $w$ is constant on both $A$ and $B$, since $w_0 = 0$ by Observation \ref{obs:0_in_slot_0}.

Consider the equivalence relation $\sim$ on $(\Z/p\Z)\setminus \{0\}$ generated as follows. We say that $i \sim j$ if $w_i = w_j$ for all $w$ which work together with $\e$ and $v$. Thus, it suffices to prove that the only two equivalence classes are $A$ and $B$. The following claim will help us to do so.
\begin{claim*}
Suppose that $i \in B$ and $j \in A$ are such that  $i+j \in A$. Then $j \sim i+j$. The same also holds with $A$ and $B$ reversed.
\end{claim*}

\begin{claimproof}[Proof of Claim]
We only prove the first case as the second is almost identical. Consider the vector $x = e_0 + e_i + e_{-j}$ of weight 3. The only shifts of $x$ which are orthogonal to $\e$ are ${x}$, $\sigma^{-i}{x}$ and $\sigma^j{x}$, and by assumption neither ${x}$ nor $\sigma^{-i}{x}$ are orthogonal to $v$.

Thus the only shift of $x$ which is orthogonal to both $\e$ and $v$ is $\sigma^j{x}$, and this must therefore be orthogonal to $w$ for all vectors $w$ that work together with $\e$ and $v$. Hence $w_j = w_{i+j}$ for all such $w$, which means that $j \sim i+j$.
\end{claimproof}

Now, we prove that the only two equivalence classes are $A$ and $B$. We will prove that $A$ is an equivalence class; the proof that $B$ is an equivalence class is almost identical. Suppose that $C \subsetneq A$ is an equivalence class of $\sim$ which is not the whole of $A$; without loss of generality, we may assume that $C$ is the smallest such class, which means that $|C| \leq |A|/2$. 

Suppose that $i \in B$ and $j \in C$. Then $i+j \neq 0$, so either $j+ i \in B$, or $j +i \in A$. In the latter case, $j \sim j+i$, and so the claim tells us that $j+i \in C$. In particular, we always have  $j + i  \notin A \setminus C$. Thus, $B + C \subseteq B \dju C$.

Now, letting $B_0 = B \dju \{0\}$ and $C_0 = C \dju \{0\}$, we see that 
\begin{equation}\label{eqn:nwws_subset}
B_0 + C_0 \subseteq B \dju C \dju \{0\}.
\end{equation} 

Since $C \neq A$, it must be the case that $|B \dju C \dju \{0\}| < p$. Thus, by Cauchy-Davenport we have that $|B_0 + C_0| \geq |B_0| + |C_0| - 1 = |B| + |C| + 1$, and so we have equality here and in (\ref{eqn:nwws_subset}). Both of $B_0$ and $C_0$ have at least two elements, so we have two cases.

\textbf{Case 1: } $|B_0 + C_0| \geq p - 1$. Then $|B| + |C| + 1 = |B_0 + C_0| = p - 1$, and so $A \setminus C$ has only one element. By assumption, $|C| \leq |A|/2$, and so we must have $|A| = 2$ and $|C| = 1$.

Since $A$ is symmetric, it is of the form $\{-a,a\}$ with either $a$ or $-a$ the unique element of $C$; without loss of generality, we can assume that $C = \{1\}$ and $A = \{-1,1\}$. Since $p>3$, we must have $2 \in B$, and applying the Claim with $i=2$ and $j=-1$ gives $-1\sim 1$. This contradicts the fact that $C$ is an equivalence class.

\textbf{Case 2:} $|B_0 + C_0| \leq p-2$, and so we can apply Vosper's Theorem. This tells us that $B_0$ and $C_0$ are arithmetic progressions with the same common difference, which we may assume is 1 by scaling. Since $0$ is part of the arithmetic progression $B_0$ and $B_0 = -B_0$, both $1$ and $-1$ must be in $B_0$. However, $C_0$ is an arithmetic progression with common difference 1, so either $1$ or $-1$ must be in $C_0$. This contradicts the fact that $C \subseteq A$ and $A$ and $B$ are disjoint.

We conclude that the only two equivalence classes must be $A$ and $B$. Therefore, any vector $w$ which works with both $\e$ and $v$ must be constant on both $A$ and $B$, which means that $w$ is contained in the span of $\e$ and $v$, as required.
\end{proof}

As well as using that $n$ is odd, the proof of Proposition \ref{prop:all_sym_work} crucially relies on the fact that $v$ is symmetric. In view of this, it is tempting to conjecture that a converse is true; in other words, that the only vectors $v \in \F_2^n$ which work together with $\e$ are the symmetric vectors. This leads us to propose Conjecture \ref{conj:our_in_intro}. Indeed, if Conjecture \ref{conj:our_in_intro} were true, then Lemma \ref{lem:wlog_0all1} and Lemma \ref{lem:nothing_works_with_sym} would tell us immediately that $h_2(p) = 2$ for any prime $p$ with 2 as a primitive root.

Note that Conjecture \ref{conj:our_in_intro} cannot hold for $p=7$, as demonstrated by the following counterexample.
\begin{observation}\label{obs:counterexample_7}
In $\F_2^7$, the following two vectors work together:
\begin{align*}
    \e &= (0,1,1,1,1,1,1) \\
    v &= (0,1,1,0,0,0,0).
\end{align*}
\end{observation}

We have verified Conjecture \ref{conj:our_in_intro} for $3 \leq n \leq 43$, and have some partial progress towards it for arbitrary primes. We will present our partial progress in the next section, and then show how we can use it to prove Theorem \ref{thm:main_theorem} in Section \ref{sec:proof}.

\section{Partial progress towards conjecture \ref{conj:our_in_intro}}
\label{sec:progress}
In this section, we show that any vector $v$ which is a counterexample to Conjecture \ref{conj:our_in_intro} must have certain properties.

\subsection{A further reformulation of the problem}
First, we provide a reinterpretation of the statement that $v$ works together with $\e$, which will help us to describe the structure of such a vector. 

To any vector $v\in \F_2^n$, we associate a Cayley (multi)digraph $G_v$ on vertex set $\Z/n\Z$, where we draw a directed edge from $i$ to $i+j$ exactly when $v_j = 1$. By Observation \ref{obs:0_in_slot_0}, we may restrict to vectors $v \in \F_2^n$ with $v_0 = 0$, which means that $G_v$ will have no self-loops. The relevance of $G_v$ to our problem is given by the following proposition.
\begin{proposition}\label{prop:graph_theory_formulation}
Let $n$ be odd. 
A vector $v\in \F_2^n$ with $v_0=0$ works together with $\e$ if and only if $G_v$ has no induced subgraph on an odd number of vertices where each vertex has odd outdegree.
\end{proposition}

\begin{proof}
Suppose first that $v$ and $\e$ do not work together. Equivalently, there exists a vector $x \in \F_2^n$ such that any shift $\sigma^k{x}$ is non-orthogonal to at least one of $\e$ and $v$. Note that $v$ must have even weight, else the vertices of $G_v$ all have odd outdegree and $G_v$ itself is our required subgraph. In particular, we may assume that $x$ has odd weight. 

Let $A = \{i : x_i = 1\}$, and consider $G_v[A]$. We have that $\sigma^{-k}x$ is orthogonal to $\e$ exactly when $x_{k} = 1$. Thus, for each $k$ with $x_{k} = 1$, we must have $v \cdot \sigma^{-k}x = 1$. Now, note that 
\[
v \cdot \sigma^{-k}x = \sum_{i : x_i = 1} v_{i-k}= \sum_{i\in A} v_{i-k}=\sum_{i \in A: i, i-k\in A} 1,
\]
which is exactly the outdegree of $k$ in $G_v[A]$ modulo 2. Thus, $G_v[A]$ is our required subgraph.

Conversely, suppose that $A$ is a set with odd size such that every vertex in $G_v[A]$ has odd outdegree. Define the vector $x$ by $x = \sum_{i \in A} e_i$. As $A$ has odd size, $\e \cdot \sigma^{-k} x = 0$ if and only if $k \in A$. But for such a $k$, $v \cdot \sigma^{-k}x = 1$ as it is the outdegree of $k$ in $G_v[A]$ modulo 2. Thus, no shift of $x$ is orthogonal to both $\e$ and $v$, proving that $\e$ and $v$ do not work together.
\end{proof}

\begin{figure}
    \centering
    \begin{tikzpicture}
    \node[font=\bfseries] (0) at (90:2) {0};
    \node[font=\bfseries]  (1) at (90- 360/9:2) {1};
    \node[font=\bfseries]  (2) at (90 - 2*360/9:2) {2};
    \node (3) at (90 - 3*360/9:2) {3};
    \node[font=\bfseries] (4) at (90 - 4*360/9:2) {4};
    \node (5) at (90 - 5*360/9:2) {5};
    \node (6) at (90 - 6*360/9:2) {6};
    \node (7) at (90 - 7*360/9:2) {7};
    \node[font=\bfseries] (8) at (90 - 8*360/9:2) {8};
    
    \draw[->, thick] (0) -- (1);
    \draw[->] (0) -- (6);

    \draw[->, thick] (1) -- (2);
    \draw[->] (1) -- (7);

    \draw[->] (2) -- (3);
    \draw[->, thick] (2) -- (8);

    \draw[->] (3) -- (4);
    \draw[->] (3) -- (0);

    \draw[->] (4) -- (5);
    \draw[->, thick] (4) -- (1);

    \draw[->] (5) -- (6);
    \draw[->] (5) -- (2);

    \draw[->] (6) -- (7);
    \draw[->] (6) -- (3);

    \draw[->] (7) -- (8);
    \draw[->] (7) -- (4);

    \draw[->, thick] (8) -- (0);
    \draw[->] (8) -- (5);
    \end{tikzpicture}
    \caption{The Cayley graph $G_v$ for $v=(0,1,0,0,0,0,1,0,0)$. The graph induced on $\{0,1,2,4,8\}$ is of odd size and every vertex has odd outdegree so the vector $(1,1,1,0,1,0,0,0,1)$ shows that $\e$ and $v$ do not work together.}
    \label{fig:example_bad_subgraph}
\end{figure}
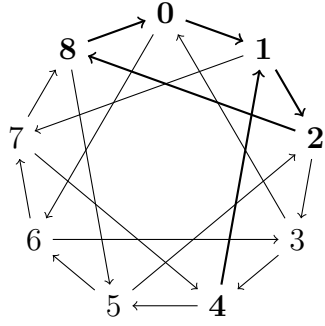

\begin{remark*}
Proposition \ref{prop:graph_theory_formulation} gives an easy proof of Proposition \ref{prop:all_sym_work}. Indeed, if $v$ is symmetric, then $G_v$ is a proper graph, and so, by the Handshaking Lemma, any subgraph with an odd number of vertices must have a vertex with even degree (in the subgraph).
\end{remark*}

By the following lemma, outdegree can also be replaced with indegree.
\begin{lemma}\label{lem:from_in_to_out}
The graph $G_v$ has an odd induced subgraph where each vertex has odd indegree if and only if it has an odd induced subgraph where each vertex has odd outdegree.
\end{lemma}

\begin{proof}
If there is an edge from $i$ to $j$ in $G_v[A]$, then there is an edge from $-j$ to $-i$ in $G_v[-A]$. Hence, passing from $G_v[A]$ to $G_v[-A]$ swaps the indegrees and the outdegrees (and the number of vertices remains odd).
\end{proof}

We call an induced subgraph $H$ of $G_v$ a \textit{bad subgraph} if
\begin{enumerate}
    \item $H$ has an odd number of vertices and
    \item at least one of the following holds:
    \begin{enumerate}
        \item each vertex in $H$ has odd outdegree within $H$;
        \item each vertex in $H$ has odd indegree within $H$.
    \end{enumerate}
\end{enumerate}
In particular, an odd induced cycle is a bad subgraph.

Proposition \ref{prop:graph_theory_formulation} and Lemma \ref{lem:from_in_to_out} imply that for $v \in \F_2^n$ with $v_0=0$, $v$ works with $\e$ if and only if $G_v$ has no bad subgraphs. An example of a bad subgraph is given in Figure \ref{fig:example_bad_subgraph}.

In what follows, we will assume that certain non-symmetric vectors work together with $\e$ and will try to find a contradiction by finding a bad subgraph. 
\subsection{Restricting the edges of the digraph}
For the remainder of the argument, we will work in $\F_2^p$, where $p$ is an odd prime. We will be searching for a bad subgraph in $G_v$, where $v \in \F_2^p$ is a non-symmetric vector with $v_0 = 0$. Given such a $v$ that works with $\e$, $G_v$ is a directed graph which may include two-way edges. We will first show that we may assume that this does not happen, which means that $G_v$ is a proper digraph.
\begin{lemma}\label{lem:no_twoway_edges}
Suppose that $v \in \mathbb{F}_2^p$ is non-symmetric with $v_0 = 0$. Suppose that there is simultaneously $i$ with $v_i = v_{-i} = 1$ and $j\neq 0$ with $v_j = v_{-j} = 0$. Then $G_v$ has a bad subgraph.
\end{lemma}

\begin{proof}
Suppose that $v \in \F_2^p$ works with $\e$; in other words, $G_v$ has no bad subgraphs. Define the sets $A = \{i : v_i = v_{-i} = 1\}$ and $B = \{j \neq 0 : v_j = v_{-j} = 0\}$. The condition on $v$ is therefore that both $A$ and $B$ are nonempty.

To prove Lemma \ref{lem:no_twoway_edges}, we will use the following claim.
\begin{claim*}
$A + B \subseteq A \dju B$.
\end{claim*}
\begin{claimproof}[Proof of Claim]
Suppose that $i \in A$ and $j \in B$. We must show that $i + j \in A \dju B$. Assume this is not the case and consider the subgraph of $G_v$ induced on $\{-i, 0, j\}$. By our assumptions, there is no edge between $0$ and $j$, a two-way edge between $0$ and $-i$ and a one-way edge between $-i$ and $j$. 

There are two cases, depending on the direction of the edge between $-i$ and $j$. If it is directed from $-i$ to $j$, then each vertex has indegree 1. Otherwise, each vertex has outdegree 1. In either case, $G_v[\{-i,0,j\}]$ is a bad subgraph.
\end{claimproof}

Now, we have $A + B \subseteq A \dju B$, which means that $A_0 + B_0 \subseteq A \dju B \dju \{0\}$, where $A_0 = A \dju \{0\}$ and $B_0 = B \dju \{0\}$. We know that $|A_0| > 1$ and $|B_0| > 1$ since $A$ and $B$ are assumed to be nonempty. Furthermore, $|A_0 + B_0| \leq |A| + |B| + 1 \leq p-2$, since otherwise $A \dju B \dju \{0\} = \Z/p\Z$ and $v$ is symmetric.

Thus, Vosper's Theorem allows us to conclude that $A_0$ and $B_0$ are both arithmetic progressions with the same common difference $d$. Since $A = -A$ and $B = -B$ are both symmetric, we deduce that $\{d, -d\}$ is contained in both $A_0$ and $B_0$, which contradicts the fact that they are disjoint.
\end{proof}

If $v$ is a vector such that there is at least one value of $i$ with $v_i = v_{-i} = 1$, we will say that $v$ is \emph{large}. Otherwise, $v$ is \emph{small}. In particular, $v$ is small if and only if $G_v$ is a proper digraph. Lemma \ref{lem:no_twoway_edges} tells us that if $v$ is non-symmetric and works with $\e$, either $v$ or $\e + v$ must be small. 

Consequently, we will only consider small vectors $v$. Given such a $v$, we will look for bad subgraphs of $G_v$. We will do this by looking at the shortest cycle in $G_v$ (throughout this paper, we use cycle to refer to a directed cycle). This cycle must have length at least 3 because $G_v$ is a proper digraph, and must be induced else $G_v$ would have a shorter cycle. We will say that the \emph{girth} of $G_v$, or (abusing notation) the \emph{girth} of $v$, is the length of the shortest cycle in $G_v$.

For the results below, we assume the following standard notation.
\begin{assump*}
We will use $v$ to denote a nonzero small vector in $\mathbb{F}_2^p$ with $v_0 = 0$ that works together with $\e$. Let $A = A^{(v)}$ be the set $\{i : v_i = 1\}$. Since $v$ is small, $A$ and $-A$ are disjoint. Let $A_0 = A^{(v)}_0$ be $A \dju \{0\}$ and let $k$ be the girth of $G_v$.
\end{assump*} 
We may immediately make some simple observations.

\begin{lemma}\label{lem:trivial_facts_about_k}
Using the standard notation, the following hold.
\begin{enumerate}
    \item $k$ is even.
    
    \item $0 \notin (k-2)A_0 + A$, or in other words
    \begin{equation}\label{eqn:bound_from_girth}
        |(k-2)A_0 + A| \leq p-1.
    \end{equation}
\end{enumerate}
\end{lemma}

\begin{proof}\leavevmode
\begin{enumerate}
    \item The shortest cycle in $G_v$ is an induced subgraph in which each vertex has indegree and outdegree 1. Thus, as $v$ and $\e$ work together, such a subgraph must have an even number of vertices, which means that $k$ must be even.
    
    \item If 0 can be written as a sum of $k-1$ elements of $A_0$, not all zero, then we get a  corresponding cycle of length at most $k-1$ in $G_v$, which we assume does not exist. The inequality (\ref{eqn:bound_from_girth}) follows immediately. \qedhere
\end{enumerate}
\end{proof}
The next proposition gives bounds on the size of $A$.
\begin{proposition}\label{prop:set_is_large}
Using the standard notation given above,
\begin{equation}\label{eqn:bounds_on_set_size} 
\frac{p}{k} < |A| < \frac{p}{k-1}.
\end{equation}
\end{proposition}

The proof of Proposition \ref{prop:set_is_large} relies on the following lemma.
\begin{lemma}\label{lem:at_least_two}
Suppose that $C$ is a $k$-cycle in $G_v$ where $v$ and $k$ are as given in the standard notation. For any $i \not \in C$, there are at least two edges between $i$ and $C$ (ignoring direction).
\end{lemma}

\begin{proof}
If there is one edge from $i$ to $C$, then $C \dju \{i\}$ induces a bad subgraph as each vertex has outdegree 1. If there is one edge from $C$ to $i$, then $C \dju \{i\}$ induces a bad subgraph as each vertex has indegree 1. In either case, we find a bad subgraph, contradicting the assertion that $v$ and $\e$ work together.

Suppose instead that there are no edges at all between $i$ and $C$. Without loss of generality (since $G_v$ is vertex transitive), we may assume $0 \in C$ and write $C$ as $C = \{0, a_1, a_1 + a_2, \dots, a_1 + \dotsb + a_{k-1}\}$ for some $a_1, \dots, a_k \in A$ with $a_1 + \dotsb + a_k = 0$.

Consider doing a single \emph{swap} of $a_1$ with $a_2$; that is, we replace the order of the edges labelled $a_1$ and $a_2$ in the cycle $C$ to get a new cycle $C_1$. Thus, $C_1 = C \setminus \{a_1\} \cup \{a_2\}$. Observe that the subgraph induced on $C_1$ is also a $k$-cycle, as any other edges between vertices would create a shorter cycle, and this does not exist by the assumption that $C$ is a shortest cycle. Note that this swap changes the number of edges between $i$ and the cycle by at most one, since at most one vertex is changed. There cannot be an edge between the new vertex $a_2$ and $i$ since that would put us in the earlier case, so we find that $i$ is not connected to $C_1$. 

We can continue doing this; swap $a_1$ with $a_3$ to form $C_2 = \{0, a_2, a_2+a_3, a_1+a_2+a_3, \dots, a_1 + \dotsb + a_{k-1}\}$, and again $i$ can have no edges to or from $C_2$. After $k-1$ iterations, we form the cycle $C_{k-1} = \{0, a_2, a_2 + a_3, \dots, a_2 + a_3 + \dotsb + a_k\}$, and there are no edges between $i$ and $C_{k-1}$. Since $\sum_{i} a_i = 0$, we can rewrite this as $C_{k-1} = \{a_1 - a_1, a_1 + a_2 - a_1, \dots, a_1 + a_2 + \dotsb + a_{k-1} - a_1, 0 - a_1\}$ so $C_{k-1} = C - a_1$ (having shifted any fixed starting point).

If we consider $C_{k-1}$ as starting at $-a_1$, the first edge to $0$ again corresponds to adding $a_1$ and we can repeat this procedure, but starting with $C_{k-1}$. This gives the cycle $C_{2(k-1)} = C - 2 a_1$, and we can continue iterating. Since $p$ is prime and $a_1 \not = 0$, we will see the family of cycles $\{ C - j : j \in \Z/p\Z\}$, and in particular every point, is in some cycle. 

However, throughout this procedure, $i$ will never be connected to any point in any of these cycles. This contradicts the assertion that every element of $\Z/p\Z$ is contained in some cycle, in view of the fact that $i$ is connected to every element of $i + A$.

Thus, there must be at least two edges between $i$ and $C$.
\end{proof}

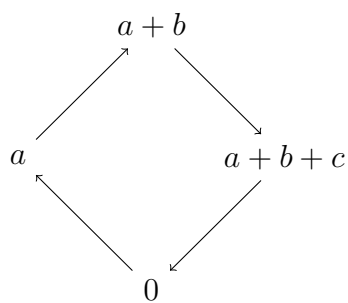
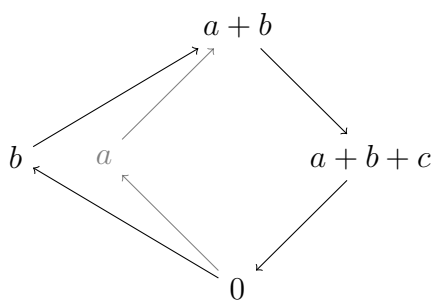
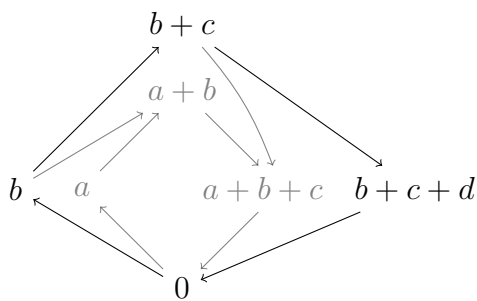
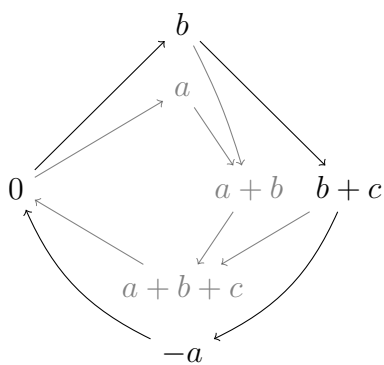
\begin{figure}
    \centering
    \begin{subfigure}[b]{0.5\textwidth}
        \begin{tikzpicture}[scale=\textwidth/6cm]
         \clip (-3, -2) rectangle (3, 3);
            \node (0) at (0,-1.5) {$0$};
            \node (a) at (-1.5,0) {$a$};
            \node (a+b) at (0,1.5) {$a+b$};
            \node (a+b+c) at (1.5,0) {$a+b+c$};
            
            \draw[->] (0) -- (a);
            \draw[->] (a) -- (a+b);
            \draw[->] (a+b) -- (a+b+c);
            \draw[->] (a+b+c) -- (0);
        \end{tikzpicture}
        \caption{The induced 4-cycle $C$.}
    \end{subfigure}
    \begin{subfigure}[b]{0.5\textwidth}
    \begin{tikzpicture}[scale=\textwidth/6cm]
     \clip (-3, -2) rectangle (3, 3);
        \node (0) at (0,-1.5) {0};
        \node (b) at (-2.5, 0) {$b$};
        \node[opacity=0.5] (a) at (-1.5,0) {$a$};
        \node (a+b) at (0,1.5) {$a+b$};
        \node (a+b+c) at (1.5,0) {$a+b+c$};
        
        \draw[->, opacity=0.5] (0) -- (a);
        \draw[->, opacity=0.5] (a) -- (a+b);
        \draw[->] (a+b) -- (a+b+c);
        \draw[->] (a+b+c) -- (0);
        \draw[->] (0) -- (b);
        \draw[->] (b) -- (a+b);
    \end{tikzpicture}
    \caption{The cycle $C_1$ formed by swapping $a$ with $b$.}
    \end{subfigure}
    \begin{subfigure}[b]{0.5\textwidth}
    \begin{tikzpicture}[scale=\textwidth/8cm]
     \clip (-3, -2) rectangle (5, 3);
        \node (0) at (0,-1.5) {0};
        \node (b) at (-2.5, 0) {$b$};
        \node[opacity=0.5] (a) at (-1.5,0) {$a$};
        \node[opacity=0.5] (a+b) at (0,1.5) {$a+b$};
        \node[opacity=0.5] (a+b+c) at (1.5,0) {\hspace{-5mm}$a+b+c$};
        \node (b+c) at (0,2.5) {$b+c$};
        \node (b+c+d) at (3.5,0) {$b+c+d$};
        
        \draw[->, opacity=0.5] (0) -- (a);
        \draw[->, opacity=0.5] (a) -- (a+b);
        \draw[->, opacity=0.5] (a+b) -- (a+b+c);
        \draw[->, opacity=0.5] (a+b+c) -- (0);
        \draw[->] (0) -- (b);
        \draw[->, opacity=0.5] (b) -- (a+b);
        \draw[->] (b) -- (b+c);
        \draw[->, opacity=0.5] (b+c) to [bend left=10] (a+b+c);
        \draw[->] (b+c) -- (b+c+d);
        \draw[->] (b+c+d) -- (0);
    \end{tikzpicture}
    \caption{The cycle $C_3 = C - a$.}
    \end{subfigure}
        \begin{subfigure}[b]{0.5\textwidth}
    \begin{tikzpicture}[scale=\textwidth/8cm]
     \clip (-4, -3) rectangle (4, 3);
        \node (0) at (-2.5,0) {0};
        \node (b) at (0, 2.5) {$b$};
        \node[opacity=0.5] (a) at (0,1.5) {$a$};
        \node[opacity=0.5] (a+b) at (1,0) {$a+b$};
        \node[opacity=0.5] (a+b+c) at (0,-1.5) {$a+b+c$};
        \node (b+c) at (2.5,0) {$b+c$};
        \node (b+c+d) at (0,-2.5) {$-a$};
        
        \draw[->, opacity=0.5] (0) -- (a);
        \draw[->, opacity=0.5] (a) -- (a+b);
        \draw[->, opacity=0.5] (a+b) -- (a+b+c);
        \draw[->, opacity=0.5] (a+b+c) -- (0);
        \draw[->] (0) -- (b);
        \draw[->, opacity=0.5] (b) to [bend left=5] (a+b);
        \draw[->] (b) -- (b+c);
        \draw[->, opacity=0.5] (b+c) -- (a+b+c);
        \draw[->] (b+c) to [bend left =20] (b+c+d);
        \draw[->] (b+c+d) to [bend left =20] (0);
    \end{tikzpicture}
    \caption{The cycle $C_3 = C - a$ after changing the point of view.}
    \end{subfigure}
    \caption{Various snapshots of $C$ throughout the swapping process.}
\end{figure}

\begin{proof}[Proof of Proposition \ref{prop:set_is_large}]
The upper bound follows immediately from (\ref{eqn:bound_from_girth}) and $k-2$ applications of the Cauchy-Davenport inequality.

For the lower bound, we will use Lemma \ref{lem:at_least_two}. Let $C$ be the vertex set of a cycle of length $k$ in $G_v$. We will count the number of pairs $(i, j)$ with $i \in \Z/p\Z$ and $j \in C$ such that there is an edge between $i$ and $j$. Each element of $C$ has $|A|$ outedges and $|A|$ inedges. Thus, there are $2|A|$ pairs for each $j \in C$, and so $2k|A|$ pairs in total. 

On the other hand, Lemma \ref{lem:at_least_two} tells us that there are at least 2 edges for each $i \in \Z/p\Z$, and so there are at least $2p$ edges. Therefore, $2p \leq 2k|A|$, which rearranges to give the lower bound claimed in Proposition \ref{prop:set_is_large}. 

To see the strict inequality, observe that $k$ is not a factor of $p$; indeed, if $k = p$ then $C = \Z/p\Z$ induces an odd cycle.
\end{proof}

\subsection{Eliminating large girth}
We will continue to restrict the possible small non-symmetric vectors $v$ that can work together with $\e$. Throughout this section we will assume the standard notation. Recall that, $A = \{i : v_i = 1\}$ and $A_0 = A \dju \{0\}$, and since $v$ is small, $A$ and $-A$ are disjoint. Denote again $G_v$ for the Cayley graph generated by $A$ as defined before Proposition \ref{prop:graph_theory_formulation}. The main result of this subsection is an upper bound on the girth $k$ of $G_v$. 
\begin{proposition}\label{prop:girth_is_small}
If a small $v \neq 0$ works together with $\e$, then the girth of $v$ is 4 or 6.
\end{proposition}

Our proof of Proposition \ref{prop:girth_is_small} will rely on a weak version of a theorem due to Grynkiewicz.
\begin{theorem}[{\cite[Theorem 19.3]{GrynkBook}}]\label{thm:grynk_thm}
Let $p$ be a prime, let $\alpha \in (0, 0.45695]$ and let $X \subseteq \Z/p\Z$ be a nonempty subset such that $|2X| = 2|X| - 1 + r < p$. Suppose that the following two bounds hold:
\begin{align}
    |2X| &\leq (2+\alpha)|X| - 3, \label{eqn:first_bound_in_grynk}\\
    |2X| &\leq \min \left\{ \frac{p+3}{2}, \frac{9 - \left(1 + 2\alpha\right)^2(2+\alpha)}{9 - \left(1 + 2\alpha\right)^2}p \right\}.
    \label{eqn:second_bound_in_grynk}
\end{align}
Then, $X$ is contained in an arithmetic progression of length at most $|X| + r$.
\end{theorem}

\begin{remark*}
Theorem \ref{thm:grynk_thm} is similar to Freiman's $3k-3$ theorem \cite[Theorem 5.11]{TaoVu}, except in $\Z/p\Z$ rather than $\Z$ (and with slightly altered conditions). In Freiman's $3k-3$ theorem, the conditions (\ref{eqn:first_bound_in_grynk}) and (\ref{eqn:second_bound_in_grynk}) are not needed; instead, the conclusion is satisfied by any set $A$ with $|2A| < 3|A| - 3$. It is conjectured that a similar weakening of the conditions should hold in $\Z/p\Z$ (see \cite[Conjecture 19.2]{GrynkBook}). 

There are various similar results in the literature which also provide weak forms of Freiman's $3k-3$ theorem in $\Z/p\Z$; for instance, Serra and Z\'{e}mor prove in \cite{SerraZemor} that the size constraint \eqref{eqn:second_bound_in_grynk} can be omitted, at the expense of a worse bound on the doubling $|2X| \leq (2 + \varepsilon)|X|$ for some small constant $\varepsilon$ (and sufficiently large $p$). Such a result could be used to obtain a version of Proposition \ref{prop:girth_is_small} with a weaker bound on the girth.
\end{remark*}

In particular, we obtain the following instance of Theorem \ref{thm:grynk_thm} by taking $\alpha = 1/3$. 
\begin{corollary}\label{cor:grynk_cor}
Let $p$ be a prime, and let $X \subseteq \Z/p\Z$ be a nonempty subset. Suppose that
\begin{align}
    |2X| &\leq \frac{7}{3}|X| - 3 \label{eqn:first_bound_in_grynk_cor}\\
    |2X| &\leq \frac{17}{42}p \label{eqn:second_bound_in_grynk_cor}.
\end{align}
Then $X$ is contained in an arithmetic progression with $|2X| - |X| + 1$ terms.
\end{corollary}

We will also need a strong upper bound on the value of $k$, depending on $p$, in the case that $A_0$ is not an arithmetic progression.
\begin{lemma}\label{lem:improved_lower_bound}
If $A_0$ is not an arithmetic progression, then $k \leq \sqrt{p+2}$.
\end{lemma}

\begin{proof}
Suppose that $A_0$ is not an arithmetic progression. Note that this implies $|A|\geq 2$. 

Equation (\ref{eqn:bound_from_girth}) tells us that $|(k-2)A_0 + A| \leq p-1$. The Cauchy-Davenport inequality gives $|(k-2)A_0| \leq p-|A|$, and thus a repeated application of Vosper's theorem gives us that $(k-2) |A_0|\leq |(k-2) A_0| \leq p-|A|$.

Write $p = rk + s$, where $0 < s < k$ and $r \geq 0$. We have that $|A| \geq r+1$ by Proposition \ref{prop:set_is_large}, and so the above inequality $(k-2) |A_0| \leq p-|A|$ gives $(k-2)(r+2) \leq (k-2)|A_0| \leq p - |A| \leq (rk+s) - (r+1)$. This rearranges to give
\begin{equation}\label{eqn:improved_lower_bound}
    2k - 3 \leq r + s.
\end{equation}

If $r < k-1$, then this equation gives $r = k-2$, $s = k-1$, and hence $p = k^2 - k -1$. The only value of $|A|$ satisfying Proposition \ref{prop:set_is_large} is $|A| = k-1$, which is odd because $k$ is even. However, $|A|$ must be even, else the whole of $G_v$ is a bad subgraph, so we must have $r \geq k-1$. If $r \geq k$, we have $p \geq k^2$. Else we have $r = k-1$ and equation \ref{eqn:improved_lower_bound} gives $s \geq k - 2$, and so $p \geq k^2 - 2$.
\end{proof}

Finally, we will need to use the following theorem due to Hamidoune and R{\o}dseth \cite{hamidoune}, which is an extension of Vosper's Theorem. Say that an \emph{almost progression} is an arithmetic progression, possibly missing one term. 
\begin{theorem}[{\cite[Theorem 3]{hamidoune}}]\label{thm:hamidoune}
Suppose that $A, B \subseteq \Z/p\Z$ are sets with $|A|, |B| \geq 3$. Suppose further that 
\[
7 \leq |A+B| \leq |A| + |B| \leq p-4.
\]
Then $A$ and $B$ are both almost progressions with the same common difference.
\end{theorem}

Corollary \ref{cor:grynk_cor} and Theorem \ref{thm:hamidoune} are used to prove the following result; the proof is rather technical and much of the additional work is due to fighting over the $-1$'s from Cauchy-Davenport.
\begin{lemma}
\label{lem:contained_in_small_AP}
If a small $v\neq 0$ works together with $\e$, and the girth of $v$ is at least 8, then $A_0$ is contained in an arithmetic progression of length at most $\frac{p}{k-2}+1$.
\end{lemma}
\begin{proof}
Our proof will fall into three cases, depending on the structure of $A_0$ and $2A_0$.

\textbf{Case 1:} \emph{$A_0$ is already an arithmetic progression.} In this case, we may use the upper bound of Proposition \ref{prop:set_is_large} to see that \[|A_0| < \frac{p}{k-1} + 1 < \frac{p}{k-2} + 1.\] Thus, we may take the progression $P$ to be $A_0$, and it will already be of the required length.

\textbf{Case 2:} \emph{$A_0$ is not an arithmetic progression, but $2A_0$ is an almost progression.}
First, we will use (\ref{eqn:bound_from_girth}) to obtain an upper bound on $|2A_0|$. Indeed, one application of Cauchy-Davenport tells us that $|(k-2)A_0| \leq p - |A|$, and a repeated application of Cauchy-Davenport gives that
\[
\frac{k-2}{2}(|2A_0| - 1) \leq |(k-2)A_0|.
\] 
Since $|A| > p/k$ by Proposition \ref{prop:set_is_large}, this immediately tells us that
\begin{equation}
    |2A_0| < \frac{2p(k-1)}{k(k-2)} + 1.
\end{equation}
Now, suppose without loss of generality that $2A_0$ is an almost progression with common difference 1, which we will treat as a subset of $\Z$ in the natural way. Say that the endpoints of $2A_0$ are $-r$ and $s$, for $r, s \geq 0$. In particular, $-r$ and $s$ are both elements of $2A_0$.

We immediately obtain that 
\begin{equation}\label{eqn:rs_upper_bound}
r + s = |2A_0| \leq \frac{2p(k-1)}{k(k-2)} + 1,
\end{equation}
since there is at most one element of $[-r, s]$ missing from $2A_0$.

Now, we claim that $A_0 \subseteq [-\floor{r/2}, \floor{s/2}]$. Indeed, $0 \in A_0$, so for any $x \in A_0$, $x + 0 \in 2A_0 \subseteq [-r, s]$. Therefore, $A_0 \subseteq [-r, s]$.

However, if $x \in A_0$, then $2x \in [-r, s]$. If $x \in [-r, -\floor{r/2}-1]$, then $2x \in [-2r, -r-1]$. Because $p - 2r > 2s \geq s$, which follows from (\ref{eqn:rs_upper_bound}) and the fact that $k \geq 8$, it follows that $A_0$ does not meet $[-r, -\floor{r/2} - 1]$. A similar argument rules out $[\floor{s/2}+1, s]$, and so it must be the case that $A_0 \subseteq [-\floor{r/2}, \floor{s/2}]$ as claimed.

Thus, $A_0$ is contained in a progression of length at most $\floor{r/2} + \floor{s/2} + 1$, which is at most 
\[
\frac{r+s}{2} + 1 \leq \frac{p(k-1)}{k(k-2)} + \frac{3}{2}.
\]
This will be no greater than $p/(k-2) + 1$ provided that $p \geq k(k-2)/2$, which follows from Lemma \ref{lem:improved_lower_bound}.

\textbf{Case 3:} \emph{$2A_0$ is not an almost progression.} Our plan is to prove that $A_0$ obeys the conditions (\ref{eqn:first_bound_in_grynk_cor}) and (\ref{eqn:second_bound_in_grynk_cor}). Again, we may use (\ref{eqn:bound_from_girth}) and (\ref{eqn:bounds_on_set_size}) to show that $|(k-2)A_0| \leq p - |A| < p(k-1)/k$.

Observe that $|2A_0| \geq 3$ and $|4A_0| \geq 7$, both of which follow from the assertion that $|A_0| \geq 3$ and the Cauchy-Davenport theorem. Furthermore, $|(k-2)A_0| \leq p(k-1)/k$, which is less than $p-4$ by Lemma \ref{lem:improved_lower_bound} and the fact that $k \geq 8$. Thus, we may apply the contrapositive of Theorem \ref{thm:hamidoune} $(k-4)/2$ times to show that 
\[
|(k-2)A_0| \geq \frac{k-2}{2}|2A_0| + \frac{k-4}{2},
\]
which, combined with our upper bound $|(k-2)A_0| \leq p - |A| < p(k-1)/k$, gives
\begin{equation}\label{eqn:sumset_upper_bound}
|2A_0| \leq \frac{2p(k-1)}{k(k-2)}- \frac{k-4}{k-2}.
\end{equation}

Now, we will first prove that $A_0$ satisfies (\ref{eqn:second_bound_in_grynk_cor}). Indeed, this is the case provided that
\[
\frac{2p(k-1)}{k(k-2)}- \frac{k-4}{k-2} \leq \frac{17}{42}p.
\]
Since $k \geq 8$, the left hand side is at most $7p/24$, which is less than $17p/42$.

It remains to prove that $A_0$ satisfies (\ref{eqn:first_bound_in_grynk_cor}). Suppose that this is not the case; in other words,
\begin{equation}
|2A_0| > \frac{7}{3}|A_0| - 3.
\end{equation}
In view of (\ref{eqn:sumset_upper_bound}) and Proposition \ref{prop:set_is_large}, we deduce that
\[
\frac{2p(k-1)}{k(k-2)} - \frac{k-4}{k-2} > \frac{7}{3} \left( \frac{p}{k} + 1 \right) - 3.
\]
Rearranging, this becomes
\[
p \left( \frac{k-8}{3k(k-2)} \right) < \frac{2}{3} - \frac{k-4}{k-2}.
\]
However when $k = 8$ both sides of this are 0, and when $k > 8$ the left hand side is positive and the right hand side is negative, which gives a contradiction.

Thus, $A_0$ satisfies (\ref{eqn:first_bound_in_grynk_cor}) and (\ref{eqn:second_bound_in_grynk_cor}), and so we may apply Corollary \ref{cor:grynk_cor}. This tells us that $A_0$ is contained in an arithmetic progression with $|2A_0| - |A_0| + 1$ terms, and this is at most $p/(k-2)$ by Proposition \ref{prop:set_is_large} and the bound (\ref{eqn:sumset_upper_bound}).
\end{proof}
We are now ready to prove Proposition \ref{prop:girth_is_small}: if a small $v \neq 0$ works together with $\e$, then the girth of $v$ is 4 or 6.
\begin{proof}[Proof of Proposition \ref{prop:girth_is_small}]
Suppose that $v$ is a vector with girth $k$ that works with $\e$. Suppose further that $k$ is not 4 or 6; this implies $k \geq 8$ since $k$ is even.

By Lemma \ref{lem:contained_in_small_AP}, $A_0$ is contained in an arithmetic progression of length at most $\frac{p}{k-2} + 1$.
Without loss of generality, we may assume $A_0$ is contained in an arithmetic progression $P$ of common difference 1, and that both endpoints of $P$ are elements of $A_0$. We will derive a contradiction by finding a bad subgraph. We will split the proof into two cases depending on whether or not $0$ is an endpoint of $P$.

\textbf{Case A:} \emph{$0$ is not an endpoint of $P$.} We will treat the elements of $P$ as elements of $\Z$ in the natural way. 

Let $P = [-R, S]$, where $R \neq S$ follows from the fact that $A$ and $-A$ are disjoint. Without loss of generality, assume that that $R < S$. The bound $|P| \leq \frac{p}{k-2} + 1$ tells us that $R + S \leq \frac{p}{k-2}$. Let $-r$ be the largest negative element of $A_0$ and $s$ be the smallest positive element of $A_0$. Consider $H = G_v[[-r, s]]$, the subgraph of $G_v$ induced on the interval $[-r, s]$. 

Each nonnegative vertex of $H$ has an outgoing edge corresponding to $-r$, and each nonpositive vertex of $H$ has an outgoing edge corresponding to $+s$. Consequently, each vertex of $H$ has positive outdegree. In particular, $H$ contains a cycle $C$, which we may assume to be induced by taking the shortest such cycle. In other words, $G_v$ contains an induced cycle $C$ contained in $[-r, s]$. 

If $C$ is an odd cycle, then it is the required bad subgraph. Otherwise, $C$ must be even. Translate so that $-r$ is the smallest element of $C$. We claim that $-r - S$ has exactly one outgoing edge pointing towards $C$, and no incoming edges from $C$. Indeed, there is an edge directed from $-r - S$ to $-r$ because $S$ is assumed to be in $A$.

Now, the outgoing edges from $-r - S$ go to $-r - S + A$, which is contained in $-r - S + [-R, S]$, and the incoming edges to $-r - S$ come from $-r - S - A$, which is contained in $-r - S + [-S, R]$. Because $S > R$, we have that \[(-r - S + A) \cup (-r - S - A) \subseteq -r - S + [-S, S],\] and so there will be no other edges between $-r - S$ and $C$ provided that $(-r - S + [-S, S]) \cap [-r, s] = \{-r\}$.

To see that this is the case, observe that $-r - S + [-S, S] = [-r - 2S, -r]$. Working in $\Z/p\Z$, the intervals $[-r - 2S, -r - 1]$ and $[-r, s]$ will be disjoint provided that $p - r - 2S > s$. This holds because $r + s$ and $S$ are both less that $\frac{p}{k-2} < \frac{p}{3}$.

Thus, there is exactly one outgoing edge from $-r - S$ to $C$, and no incoming edges from $C$ to $-r - S$. This means that each vertex in $C \dju \{-r - S\}$ has outdegree 1, and so this is our required bad subgraph.

\textbf{Case B:} \emph{$0$ is an endpoint of $P$.} We will treat elements of $P$ as elements of $[0, p]$ in the natural way. Let $m = \max{P}$, so $m < \frac{p}{k-2}$. 

We claim that any $r$ with the property that $m/2 \leq r \leq 3m/2$ is the sum of two elements of $A_0$. Indeed, $|A_0| > p/k + 1$, so there are fewer than $m - p/k - 1$ elements of $P \setminus A_0$. However, given any $r$ with $m/2 \leq r \leq 3m/2$, there are at least $m/4$ disjoint pairs of elements of $P$ which sum to $r$. Thus, if $r$ is not expressible as the sum of two elements of $A$, then it must be the case that
\[
\frac{m}{4} < m - \frac{p}{k} - 1.
\]
Since $m < p/(k-2)$ and $k \geq 8$, this cannot happen.

Now, choose $r$ with $m/2 \leq r \leq 3m/2$ so that $p - r$ is divisible by $m$ (this is possible since there are at least $m$ integers in the interval $[m/2,3m/2]$), and let $r_1$ and $r_2$ be two elements of $A_0$ which sum to $r$. This gives two cases depending on whether $r_1$ and $r_2$ are both nonzero or at least one is zero; we prove only the first case, as the second is easier.

We write $p = sm + r_1 + r_2$; since $v$ has girth at least 8, we must have $s \geq 6$. If $s$ is odd, then we have an odd induced cycle of length $s+2$ whose elements are the partial sums of
\[
\underbrace{m + \dots + m}_{s-1 \text{ terms}} + r_1 + m + r_2 = p.
\]
Note that the order of the elements is important, and that $m$ is chosen (and placed in between $r_1$ and $r_2$) such that the odd cycle 
\[
\{0,m,2m,\dots,(s-1)m,(s-1)m+r_1,sm+r_1\}
\]
is induced.

Otherwise, $s$ is even. By a similar argument to the one used to show that $r$ is a sum of two elements of $A_0$, we can deduce that $m$ is the sum of two elements of $A$; we write $m = m_1 + m_2$. This gives us an odd induced cycle of length $s+3$, whose elements are the partial sums of
\[
m + m_1 + m + m_2 + \underbrace{m + \dots + m}_{s-4 \text{ terms}} + r_1 + m + r_2 = p.\qedhere
\]
\end{proof}

\section{Proof of main result}
\label{sec:proof}
Suppose towards contradiction that there exists a prime $p$ with primitive root $2$ for which $h_2(p)>2$. Then there are three vectors $u,v,w\in \F_2^p$ that are linearly independent and work together. By Lemma \ref{lem:irreducible}, we may assume that $u=\e$. By Lemma \ref{lem:nothing_works_with_sym} and Proposition \ref{prop:girth_is_small}, it must be the case that $v$ and $w$ are non-symmetric vectors of girth $4$ or $6$. Recall that we call a vector $u$ small if there is no $i$ for which $u_i=u_{-i}=1$. By Lemma \ref{lem:no_twoway_edges}, we may assume $v$ and $w$ are small (possibly replacing $v$ or $w$ with $\e + v$ or $\e + w$).
We can moreover control the number of edges in the graphs $G_v$ and $G_w$ by Proposition \ref{prop:set_is_large}.

Since $v, w$ and $\e$ all work together, $v + w$ must also work with $\e$. Now, $v + w$ cannot be symmetric, else the fact that $\e, v$ and $v+w$ work together would contradict Lemma \ref{lem:nothing_works_with_sym}. 
To complete the proof of Theorem \ref{thm:main_theorem}, we will divide into cases depending on how many of $v, w$ and $v+w$ have girth 4 and how many have girth 6. For this, we use a second graph-theoretic formulation of the problem given in Proposition \ref{prop:graph_theory_formulation_two_vector}. We first need to establish that $v+w$ is also small.

Since Conjecture \ref{conj:our_in_intro} has been verified by computer for all $3 \leq n \leq 43$, we can assume that either $p > 43$ or $p = 7$. A simple brute force search confirms that $h_2(7) = 2$, and so we can assume throughout this section that $p > 43$ (although a much weaker bound would suffice).

\begin{lemma}\label{lem:no_ssl}
Suppose that $v, w \in \F_2^p$ are small vectors such that $v, w$ and $\e$ work together. Then, the vector $v+w$ must also be small.
\end{lemma}
As before, set $A^{(v)} = \{i:v_i = 1\}$ and $A^{(v)}_0 = A^{(v)}\dju\{0\}$, and define $A^{(w)}$ and $A^{(w)}_0$ similarly. Since $v$ and $w$ are small, $A^{(v)}$ is disjoint from $-A^{(v)}$, and the same holds for $A^{(w)}$.
\begin{proof}[Proof of Lemma \ref{lem:no_ssl}]
Suppose otherwise that $v + w$ is large. Since $v$ and $w$ are small, they must have girth at least 4, and so, by Proposition \ref{prop:set_is_large}, $A^{(v)}$ and $A^{(w)}$ have size at most $\floor{p /3}$. Since $A^{(v + w)} = A^{(v)} \triangle A^{(w)}$, $A^{(v + w)}$ must have size at most $2\floor{p/3}$, which means that $|A^{(\e + v + w)}| \geq p - 1 - 2 \floor{p/3}$. 

However, by assumption $\e + v + w$ is a small vector that works with $\e$, and so has girth at least 4. In particular, $|A^{(\e + v + w)}| \leq \floor{p/3}$.

Combining these two equations, we find that $(p - 1)/3 \leq \floor{p/3}$, which implies that, for some $m \geq 2$, $p = 3m+1$ and $|A^{(v)}| = |A^{(w)}| = |A^{(\e + v + w)}| = m$.

To rule out this special case, we will use Vosper's Theorem to find the structure of $v$ and show that $G_v$ has an induced $5$-cycle, which contradicts the assertion that $v$ works with $\e$. By two applications of Cauchy-Davenport, $|2A^{(v)}_0 + A^{(v)}| \geq 3m$. However, $|2A^{(v)}_0 + A^{(v)}| \leq 3m$ by (\ref{eqn:bound_from_girth}). Therefore, equality holds and, in particular, \[|A^{(v)}_0 + A^{(v)}| = |A^{(v)}_0| + |A^{(v)}| - 1 = 2m \leq p - 2.\] Hence, we may apply Vosper's Theorem to see that $A_0^{(v)}$ and $A^{(v)}$ are arithmetic progressions with the same common difference; without loss of generality, the common difference is 1, and $A_0^{(v)} = \{0, \dots, m\}$.

Now, we have an induced 5-cycle whose vertices are $\{0, m-1, m+1, 2m, 2m+2\}$. Indeed, $m-1$ and $2$ are both elements of $[m]$, and there will be no other edges provided that $2m - 2 > m$, which holds for $p > 7$.
\end{proof}

As with Proposition \ref{prop:graph_theory_formulation}, we can rephrase our problem in terms of a graph theoretical problem. Since $v + w$ is small, there cannot be any edges that have different directions in $G_v$ and $G_w$. Thus, we may encode $G_v$ and $G_w$ in a single graph $G$ with vertex set $\Z/p\Z$ using edges coloured with three colours. We do this as follows.
\begin{itemize}
    \item If there is an edge from $i$ to $i +j$ in both $G_v$ and $G_w$ (so $j \in A^{(v)} \cap A^{(w)}$), then draw a green edge from $i$ to $i + j$.
    \item If there is an edge from $i$ to $i + j$ in $G_v$ but not in $G_w$ (so $j \in A^{(v)} \setminus A^{(w)}$), then draw a red edge from $i$ to $i + j$.
    \item Finally, if there is an edge from $i$ to $i + j$ in $G_w$ but not in $G_v$ (so $j \in A^{(w)} \setminus A^{(v)}$), then draw a blue edge from $i$ to $i+j$. 
\end{itemize}
Thus, $G_v$ consists of the red and green edges of $G$, and $G_w$ consists of the blue and green edges. Furthermore, $G_{v+w}$ consists of the red and blue edges.

The relevance of $G$ to our problem is given by the following proposition, which is analogous to Proposition \ref{prop:graph_theory_formulation}.
\begin{proposition}
\label{prop:graph_theory_formulation_two_vector}
If $v, w \in \F_2^p$ are two small vectors with $v_0 = w_0 = 0$, then $v, w$ and $\e$ work together if and only if there is no induced subgraph $H$ of $G$ on an odd number of vertices such that, for each $i \in H$, the numbers of red, blue and green outneighbours of $i$ in $H$ are not all of the same parity.
\end{proposition}
As with Lemma \ref{lem:from_in_to_out}, the outdegrees can be replaced with indegrees. Thus, for the remainder of the proof, we will call an induced subgraph $H$ of $G$ \textit{bad} if
\begin{enumerate}
    \item $H$ has an odd number of vertices and
    \item at least one of the following holds:
    \begin{enumerate}
        \item for each vertex in $H$, the outdegrees in the three colours do not all have the same parity;
        \item for each vertex in $H$, the indegrees in the three colours do not all have the same parity.
    \end{enumerate}
\end{enumerate}
In other words, Proposition \ref{prop:graph_theory_formulation_two_vector} implies that $v$ and $w$ work together with $\e$ if and only if $G$ has no bad subgraphs.
\begin{proof}[Proof of Proposition \ref{prop:graph_theory_formulation_two_vector}]
The vectors $v$, $w$ and $\e$ fail to work together if and only if there is a vector $x$ for which any shift $\sigma^k x$ is not orthogonal to at least one of them. Let $B = \{i:x_i = 1\}$.

As shown in the proof of Proposition \ref{prop:graph_theory_formulation}, we find that $v \cdot \sigma^{-k} x = 1$ if and only if $k$ has odd outdegree in $G_v[B]$, and similarly with $v$ replaced by $w$.

Thus, $v, w$ and $\e$ fail to work together if and only if there is a subset $B$ of odd cardinality such that, for each $k \in B$, $k$ has odd outdegree in at least one of $G_v[B]$ and $G_w[B]$. 
Since $G_v$ consists of the red and green edges of $G$, the outdegree of some $k\in B$ is even in $G_v[B]$ if and only if the number of red and green outedges of $k$ in $G[B]$ has the same parity. The same statement holds for $w$ with red replaced by blue.
Hence the statement that $k$ has odd outdegree in at least one of $G_v[B]$ and $G_w[B]$ is equivalent to the assertion that the outdegrees of $k$ in $G[B]$ in the three colours are not all of the same parity.
\end{proof}
\begin{observation}
\label{obs:no_triangle}
For any colouring of the edges, a directed triangle is a bad subgraph.
\end{observation}
Let us write $A$ for $A^{(v)} \cup A^{(w)}$ so there is an edge from $i$ to $i+j$ in $G$ if and only if $j \in A$. Let $A_\green = A^{(v)} \cap A^{(w)}$ (so there is a green edge from $i$ to $i + j$ if and only if $j \in A_\green$) and similarly define $A_\red = A^{(v)} \setminus A^{(w)}$ and $A_\blue = A^{(w)} \setminus A^{(v)}$.

Since $v$, $w$ and $v + w$ all work together with $\e$, Proposition \ref{prop:girth_is_small} tells us they must each have girth 4 or girth 6. We will divide the remainder of the proof into cases depending on how many of $v$, $w$ and $v+w$ have girth 4 and how many have girth 6.

\begin{lemma}\label{lem:no_two_4}
Suppose that at least two of $v, w$ and $v+w$ have girth 4. Then $G$ has a bad subgraph.
\end{lemma}

\begin{proof}
Without loss of generality, we can assume that both $v$ and $w$ have girth 4. By Proposition \ref{prop:set_is_large}, $|A^{(v)}|, |A^{(w)}| > p/4$ and $|A^{(v + w)}| > p/6$. Since $A^{(v + w)} = A^{(v)} \triangle A^{(w)}$, each element $j \in A$ is in exactly 2 of $A^{(v)}$, $A^{(w)}$ and $A^{(v+w)}$. Hence, 
\begin{align*}
    |A| &= \frac{1}{2}\left(|A^{(v)}| + |A^{(w)}| + |A^{(v + w)}|\right) \\
    &> \frac{1}{2}\left(\frac{p}{4} + \frac{p}{4} + \frac{p}{6}\right)= \frac{p}{3}.
\end{align*}
Thus, the Cauchy-Davenport Theorem tells us that $|2(A \dju \{0\}) + A | \geq p$, and in particular $0$ is the sum of at most 3 elements of $A$. Thus, there is a cycle of length at most $3$ in $G$. There are no two-way edges in $G$, so $G$ contains a directed triangle.
\end{proof}

For the remaining two cases, we will require the following lemma.
\begin{figure}
    \centering
    \begin{tikzpicture}
        \node (0) at (0,0) {$0$};
        \node (r) at (4, 2) {$r$};
        \node (g) at (4,0) {$g$};
        \node (b) at (4, -2) {$b$};
        \node (rg) at (8, 2) {$r+g$};
        \node (rb) at (8,0) {$r+b$};
        \node (bg) at (8, -2) {$b+g$};
        \node (rgb) at (12,0) {$r+g+b$};
        
        \draw[->, red] (0) -- (r);
        \draw[->, red] (b) -- (rb);
        \draw[->, red] (g) -- (rg);
        \draw[->, red] (bg) -- (rgb);
        \draw[->, green] (0) -- (g);
        \draw[->, green] (r) -- (rg);
        \draw[->, green] (b) -- (bg);
        \draw[->, green] (rb) -- (rgb);
        \draw[->, blue] (0) -- (b);
        \draw[->, blue] (r) -- (rb);
        \draw[->, blue] (g) -- (bg);
        \draw[->, blue] (rg) -- (rgb);
        \draw[->, black] (rgb) to [bend left = 60] (0);
    \end{tikzpicture}
    \caption{A rainbow 4-cycle implies the existence of a subgraph as depicted in the figure. The black edge can be coloured either red, green or blue.}
    \label{fig:everything-coloured}
\end{figure}
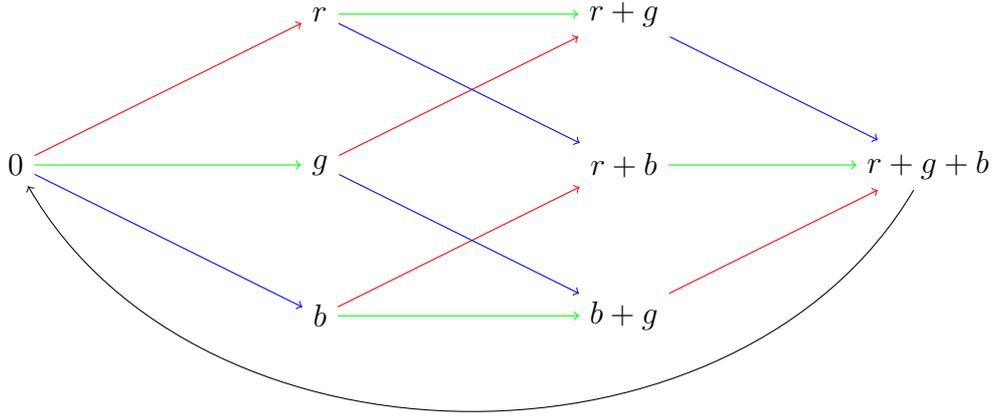

\begin{lemma}\label{lem:no_rainbow_4}
If $G$ contains a 4-cycle which includes edges of all three colours, then $G$ contains a bad subgraph.
\end{lemma}

\begin{proof}
Suppose that $G$ contains a 4-cycle which includes edges of all three colours. Label a red edge $r$ (so the edge is directed from some $i$ to $i+r$), label a blue edge $b$, a green edge $g$ and the final edge $x$. Since this proof doesn't use any properties of the individual colours, we can assume, without loss of generality, that the edges are in the order $r, g, b, x$. We will consider subgraphs induced on (subsets of) $\{0, r, g, b, r+g, r+b, b+g, r+ g + b\}$.

First, consider the subgraph $H = G[\{0,r,g, r+g, r+g+b\}]$, as illustrated by Figure \ref{fig:subset-coloured}. 
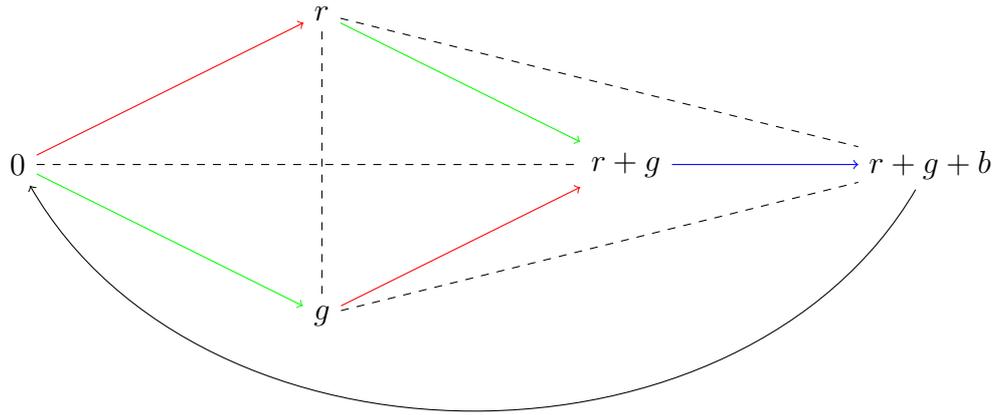
\begin{figure}[ht]
    \centering
      \begin{tikzpicture}
        \node (0) at (0,0) {$0$};
        \node (r) at (4, 2) {$r$};
        \node (g) at (4, -2) {$g$};
        \node (rg) at (8, 0) {$r+g$};
        \node (rgb) at (12,0) {$r+g+b$};
        
        \draw[->, red] (0) -- (r);
        \draw[dashed] (0) -- (rg);
        \draw[dashed] (r) -- (rgb);
        \draw[dashed] (g) -- (rgb);
        \draw[dashed] (r) -- (g);
        \draw[->, red] (g) -- (rg);
        \draw[->, green] (0) -- (g);
        \draw[->, green] (r) -- (rg);
        \draw[->, blue] (rg) -- (rgb);
        \draw[->, black] (rgb) to [bend left = 60] (0);
    \end{tikzpicture}
    \caption{If none of the dotted edges are present, then the depicted subgraph of $G$ is a bad subgraph.}
    \label{fig:subset-coloured}
\end{figure}
Observe that there can be no edge between $0$ and $r+g$. If there were an edge from $0$ to $r+g$, then $\{0, r+g, r+g+b\}$ would induce a directed triangle, and if there were an edge from $r+g$ to $0$ then $\{0, g, r+g\}$ would induce a directed triangle. In either case, we obtain a bad subgraph. Similarly, there can be no edge between $r+g+b$ and either $r$ or $g$. Thus, the only possible edge that can be added is between $r$ and $g$.

Now, observe that if there is no edge between $r$ and $g$, then $H$ is a bad subgraph. In fact, the only way to stop $H$ from being bad is to have a red edge from $g$ to $r$, or a green edge from $r$ to $g$.

The same argument can be applied to subgraphs induced by $\{0,r,b, r+b, r+g+b\}$ and $\{0, g, b, b + g, r + g + b\}$ to show that there is either a blue edge from $r$ to $b$ or a red edge from $b$ to $r$, and either a blue edge from $g$ to $b$ or a green edge from $b$ to $g$. If we take edges of all three colours, then $\{r, g, b\}$ would induce a directed triangle. Thus, one colour appears twice between pairs of $r, g$ and $b$; without loss of generality, that colour is red. We may hence assume there are red edges directed from $g$ and $b$ to $r$.

Now, consider the subgraph $H$ induced by the vertices $\{0, r, b, r+g, r + g +b\}$, as seen in Figure \ref{fig:subset-coloured-2}. 

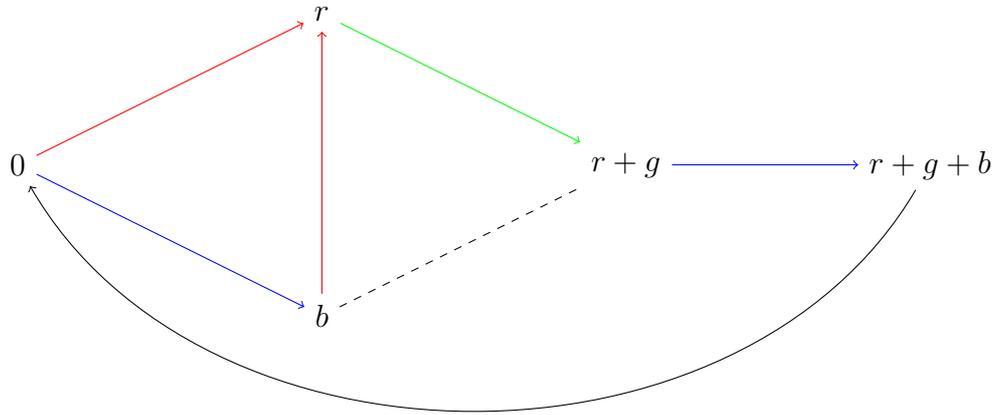
\begin{figure}
    \centering
    \begin{tikzpicture}
        \node (0) at (0,0) {$0$};
        \node (r) at (4, 2) {$r$};
        \node (b) at (4, -2) {$b$};
        \node (rg) at (8, 0) {$r+g$};
        \node (rgb) at (12,0) {$r+g+b$};
        
        \draw[->, red] (0) -- (r);
        \draw[->, red] (b) -- (r);
        \draw[dashed] (b) -- (rg);
        \draw[->, green] (r) -- (rg);
        \draw[->, blue] (0) -- (b);
        \draw[->, blue] (rg) -- (rgb);
        \draw[->, black] (rgb) to [bend left = 60] (0);
    \end{tikzpicture}
    \caption{Vertex $r$ has an even number of edges coming in for each colour. If the dotted line is not an edge, then the depicted subgraph is bad on outdegrees.}
        \label{fig:subset-coloured-2}
\end{figure}
As before, the only edge that is undetermined is a possible edge between $b$ and $r+g$; such an edge cannot be directed from $r+g$ to $b$, else $\{b, r, r+g\}$ induces a triangle.

If there is no edge from $b$ to $r+g$, then $H$ is bad. Indeed, $H$ is bad unless there is a red edge from $b$ to $r+g$.

Finally, consider the subgraph induced on $\{0, b, g, r +g, r+b+g\}$ as depicted in Figure \ref{fig:subset-coloured-3}. 

\begin{figure}
    \centering
    \begin{tikzpicture}
        \node (0) at (0,0) {$0$};
        \node (b) at (4, 2) {$b$};
        \node (g) at (4, -2) {$g$};
        \node (rg) at (8, 0) {$r+g$};
        \node (rgb) at (12,0) {$r+g+b$};
        
        \draw[->, blue] (0) -- (b);
        \draw[->, green] (0) -- (g);
        \draw[->, red] (b) -- (rg);
        \draw[->, red] (g) -- (rg);
        \draw[->, blue] (rg) -- (rgb);
        \draw[dashed] (b) -- (g);
        \draw[->, black] (rgb) to [bend left = 60] (0);
    \end{tikzpicture}
    \caption{There is no way to add an edge to the depicted subgraph without creating a bad subgraph.}
    \label{fig:subset-coloured-3}
\end{figure}
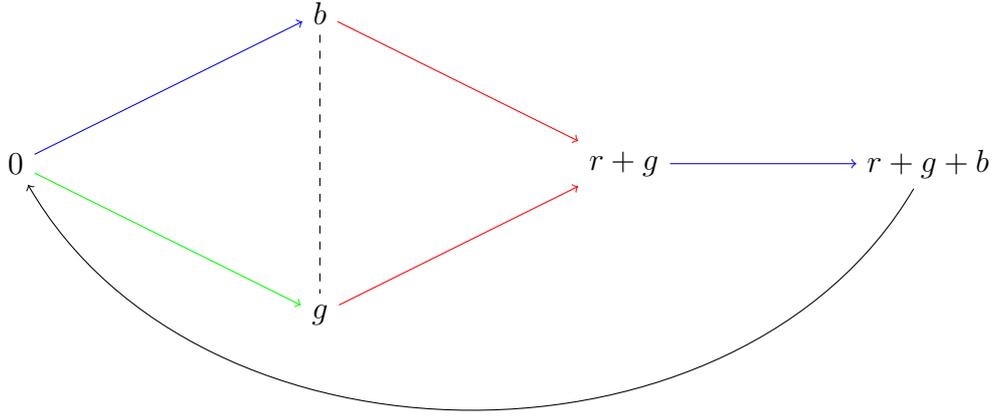
We know that there is an edge between $b$ and $g$ and that edge is either blue or green. In both cases, this is the bad subgraph we sought.
\end{proof}

\begin{lemma}\label{lem:no_666}
Suppose that $v, w$ and $v+w$ all have girth 6. Then $G$ has a bad subgraph.
\end{lemma}
\begin{proof}
As in the proof of Lemma \ref{lem:no_two_4}, we learn that $|A| > p/4$, and so, by 3 applications of the Cauchy-Davenport Theorem, $G$ contains a cycle of length at most 4. Since there can be no directed triangles (Observation \ref{obs:no_triangle}), the girth must be at least $4$, so there is an induced 4-cycle $C$. 

Suppose that $C$ contains only red and green edges; then $C$ is contained entirely within $G_{v}$, contradicting the fact that $v$ has girth 6. Therefore, $C$ must contain at least one blue edge. Similarly $C$ must contain at least one red edge and at least one green edge. Thus, $C$ must contain all three colours, and so we are done by Lemma \ref{lem:no_rainbow_4}.
\end{proof}

\begin{lemma}\label{lem:no_664}
Suppose that $v$ and $w$ have girth 6, and $v + w$ has girth 4. Then $G$ has a bad subgraph.
\end{lemma}

\begin{proof}

As in the proof of Lemma \ref{lem:no_666}, since $G_{v}$ and $G_{w}$ do not contain a 4-cycle, we find that there cannot be a 4-cycle in $G$ with no blue edges or with no red edges. Since there cannot be a 4-cycle with all three colours by Lemma \ref{lem:no_rainbow_4}, we find that $G$ cannot have a 4-cycle containing a green edge. Note that $G$ can also not contain an induced directed 3-cycle or 5-cycle, as such a subgraph is bad.

By Proposition \ref{prop:set_is_large}, $p/4<|A^{(v+w)}|<p/3$, and by Cauchy-Davenport
\[
|3A^{(v+w)}_0+A^{(v+w)}|\geq \min\{p, 4|A^{(v+w)}|\}\geq p.
\]
Thus, $3A^{(v+w)}_0+A^{(v+w)}$ is the whole of $\Z/p\Z$. Consequently, for any $g\in A_\green$, we can find $a_1,\dots,a_4\in A^{(v+w)}_0$ not all zero such that $a_1+a_2+a_3+a_4=-g$. Note that $-g\not \in A^{(v+w)}=A_\red\cup A_\blue$, so at least two of the $a_i$ are non-zero. Exactly two non-zero $a_i$ corresponds to a directed 3-cycle, whereas exactly three non-zero $a_i$ would correspond to a directed 4-cycle with a green edge, both of which cannot exist by the observations above. We conclude that all $a_i$ must be non-zero and that $G$ contains a (non-induced) 5-cycle $C$ with exactly one green edge.

The cycle $C$ must contain at least one red edge and at least one blue edge. Indeed, suppose $C$ doesn't contain a red edge. Then $C$ consists entirely of blue and green edges and is contained in $G_w$, which contradicts our assumption that the girth of $G_w$ is 6. The argument that $C$ contains a blue edge is identical. Without loss of generality there are at least two red edges, so we may assume that $r = a_1\in A_\red$, $r' = a_2 \in A_\red$ and $ b = a_3 \in A_\blue$. Set $ x= a_4 \in A_\red \cup A_\blue$.

We will now consider subgraphs induced on (subsets of) $\{0,b,b+g,b+r,b+g+r,-x\}$ as in Figure \ref{fig:two-pentagons}.
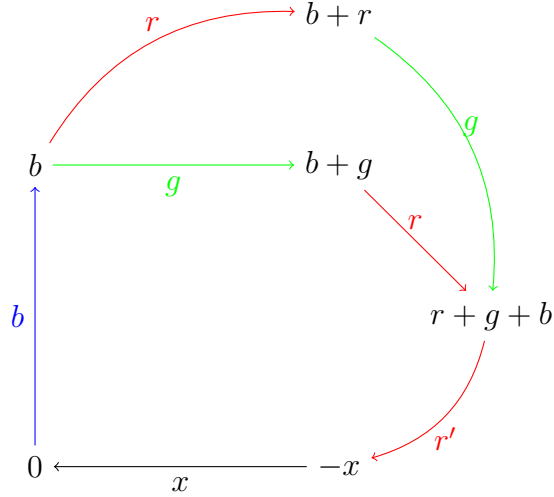
\begin{figure}
    \centering
    \begin{tikzpicture}
        \node (A) at (0,0) {$0$};
        \node (B) at (0,4) {$b$};
        \node (C) at (4,4) {$b+g$};
        \node (D) at (4,6) {$b+r$};
        \node (E) at (6,2) {$r+g+b$};
        \node (F) at (4,0) {$-x$};
\path[->]
(A) edge[blue] node[left] {$b$} (B)
(B) edge[green] node[below] {$g$} (C)
(B) edge[bend left,red] node[above] {$r$} (D)
(C) edge[red] node[above] {$r$} (E)
(D) edge[bend left,green] node[above] {$g$} (E)
(E) edge[red, bend left] node[midway,below] {$r'$} (F)
(F) edge[] node[below] {$x$} (A)
   ;
    \end{tikzpicture} 
    \caption{If the graph $G$ contains a 5-cycle, then we may assume it contains the subgraph depicted.}
    \label{fig:two-pentagons}
\end{figure}

First, consider $H = G[\{0, b, b+g, r+g+b, -x\}]$. Since $G$ has no induced 5-cycles, there must be another edge present. As shown in Figure \ref{fig:two-pentagons-extra-edges}, there are only two possible edges that can be added without creating a 3-cycle or a 4-cycle containing a green edge, from $0$ to $b + g$ and from $b$ to $r + g + b$.
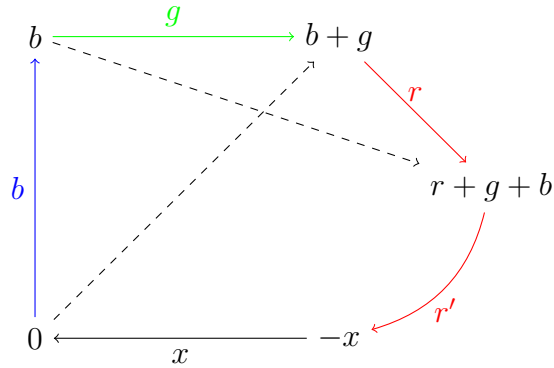
\begin{figure}
    \centering
    \begin{tikzpicture}
        \node (A) at (0,0) {$0$};
        \node (B) at (0,4) {$b$};
        \node (C) at (4,4) {$b+g$};
        \node (E) at (6,2) {$r+g+b$};
        \node (F) at (4,0) {$-x$};
\path[->]
(A) edge[blue] node[left] {$b$} (B)
(A) edge[dashed] node[] {} (C)
(B) edge[green] node[above] {$g$} (C)
(B) edge[dashed] node[] {} (E)
(C) edge[red] node[above] {$r$} (E)
(E) edge[red, bend left] node[midway,below] {$r'$} (F)
(F) edge[] node[below] {$x$} (A)
   ;
    \end{tikzpicture}
    \caption{The only edges that can be added to the subgraph without creating a rainbow 4-cycle or a directed triangle are depicted with dotted lines.}
    \label{fig:two-pentagons-extra-edges}
\end{figure}
Since $\{0, b, b+g, r+g+b, - x\}$ is an odd-sized set, $H$ will be bad unless there is a vertex for which the indegree has the same parity in every colour, and also a vertex for which the outdegree has the same parity in every colour. 

As we cannot add green edges without creating a 4-cycle with a green edge, $H$ will be bad because of indegree unless we have a red edge from $b$ to $b+g+r$. Similarly, $H$ will be bad because of outdegree unless we have a blue edge from $0$ to $b+g$.

A similar analysis on $\{0, b, b+r, r+g+b, -x \}$ shows us that there is a red edge from $b+r$ to $-x$, and so the only unknown edge is a possible edge between $b+g$ and $b+r$, as illustrated by Figure \ref{fig:two-pentagons-complete} (on the next page).
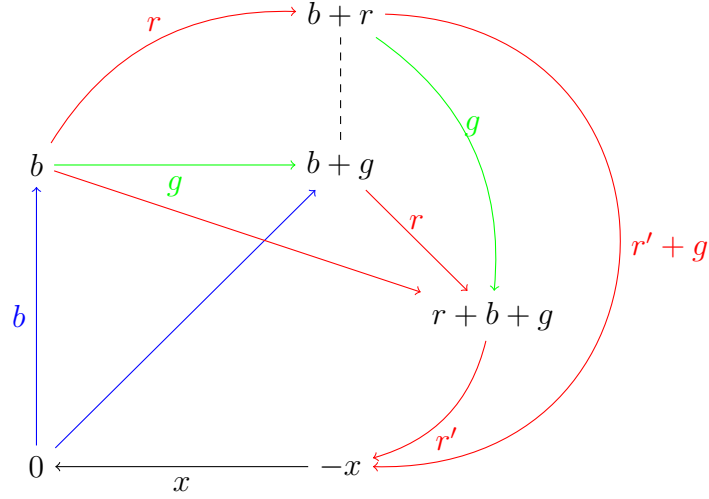
\begin{figure}
    \centering
    \begin{tikzpicture}
        \node (A) at (0,0) {$0$};
        \node (B) at (0,4) {$b$};
        \node (C) at (4,4) {$b+g$};
        \node (D) at (4,6) {$b+r$};
        \node (E) at (6,2) {$r+b+g$};
        \node (F) at (4,0) {$-x$};
\path[->]
(A) edge[blue] node[left] {$b$} (B)
(A) edge[blue] node[] {} (C)
(B) edge[green] node[below] {$g$} (C)
(B) edge[bend left,red] node[above] {$r$} (D)
(B) edge[red] node[] {} (E)
(C) edge[red] node[above] {$r$} (E)
(C) edge[-,dashed] node[] {} (D)
(D) edge[bend left,green] node[above] {$g$} (E)
(E) edge[red, bend left] node[midway,below] {$r'$} (F)
(F) edge[] node[below] {$x$} (A)
(D.east) edge[red, bend left=90, looseness=1.8] node[right] {$r'+g$} (F.east)
   ;
    \end{tikzpicture}  
    \caption{The dotted line indicates the only place where we might be able to add an edge to the subgraph.}
    \label{fig:two-pentagons-complete}
\end{figure}
We make the following two claims:
\begin{enumerate}
    \item There must be either a red edge from $b+g$ to $b+r$ or a red edge from $b+r$ to $b+g$.
    \item There must be either a red edge from $b+g$ to $b+r$ or a blue edge from $b+r$ to $b+g$.
\end{enumerate}
To see the first claim, consider $H = G[\{0,b,b+g,b+r,-x\}]$. The only way that $H$ can have a vertex all of whose indegrees have the same parity is if there is a red edge from $b+g$ to $b+r$ or a red edge from $b+r$ to $b+g$. Similarly, the second claim follows by considering the outdegrees of $G[\{0, b+g, b+r, b+g+r,-x\}]$.

The only way both 1 and 2 can hold is if there is a red edge from $b+g$ to $b+r$. In particular, $r - g \in A_\red$.
\begin{figure}
    \centering
    \begin{tikzpicture}
        \node (A) at (0,0) {$0$};
        \node (B) at (0,4) {$b$};
        \node (C) at (4,4) {$b+g$};
        \node (D) at (4,6) {$b+r$};
        \node (E) at (6,2) {$r+g+b$};
        \node (F) at (4,0) {$-x$};
\path[->]
(A) edge[blue] node[left] {$b$} (B)
(A) edge[blue] node[] {} (C)
(B) edge[green] node[below] {$g$} (C)
(B) edge[bend left,red] node[above] {$r$} (D)
(B) edge[red] node[] {} (E)
(C) edge[red] node[above] {$r$} (E)
(C) edge[red] node[left] {$r-g$} (D)
(D) edge[bend left,green] node[above] {$g$} (E)
(E) edge[red, bend left] node[midway,below] {$r'$} (F)
(F) edge[] node[below] {$x$} (A)
(D.east) edge[red, bend left=90, looseness=1.8] node[right] {$r'+g$} (F.east)
   ;
    \end{tikzpicture}  
    \caption{Starting from the 5-cycle depicted in Figure \ref{fig:two-pentagons-extra-edges}, we show $G$ either contains a bad subgraph or it contains the subgraph depicted in the figure.}
    \label{fig:two-pentagons-everything}
\end{figure}
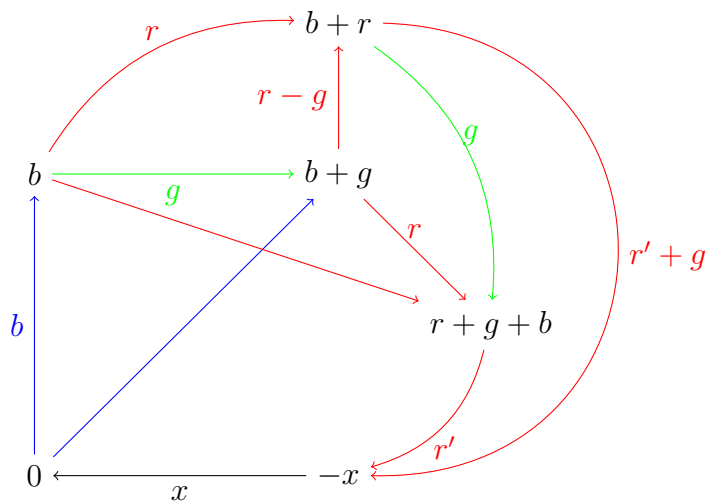

Thus, given a 5-cycle $b + r + g + r' + x = 0$, we get that $r-g$ and $r'+g$ are red, and $b + (r-g) + g + (r' + g) + x = 0$ is also a 5-cycle. We may iterate this argument to see that each element of the sequence $r, r-g, r-2g, \dots$ will be red. However, $p$ is prime, which means that the sequence will hit every element of $\Z/p\Z$. This yields the required contradiction, and so $G$ must have a bad subgraph.
\end{proof}

Lemmas \ref{lem:no_two_4}, \ref{lem:no_666} and \ref{lem:no_664} combine to show that regardless of the girths of $v$ and $w$, $G$ must have a bad subgraph, and so $v$ and $w$ cannot work together with $\e$. This completes the proof of Theorem \ref{thm:main_theorem}.

\section{Generalisation to arbitrary finite fields}
\label{sec:hqn}
We will now investigate $h_q(n)$, the maximal codimension of a cyclically covering subspace in $\F_q^n$, for $q$ a prime power. We consider this problem in the orthogonal complement in an analogous way to the case $q=2$: 
we say that a vector $v \in \mathbb{F}_q^n$ \emph{works} if for every $x \in \mathbb{F}_q^n$ there is a $k\in \{0,\dots, n-1\}$ such that $v \cdot \sigma^k x = 0$, and that the vectors $v^{(1)}, v^{(2)}, \dotsb , v^{(m)}$ \emph{work together} if for every $x \in \mathbb{F}_q^n$ there exists a $k$ such that \[ v^{(1)} \cdot \sigma^k x = v^{(2)} \cdot \sigma^k x = \dotsb = v^{(m)} \cdot \sigma^k x = 0 .\] 
As before, $h_q(n)$ is given by the largest $m$ for which there exist $v^{(1)}, v^{(2)}, \dots, v^{(m)}$ that are linearly independent and work together. 

The following result is an analogue of Theorem \ref{thm:sum_mult_bound}, and can be proven by replacing $2$ with $q$ in the proof of Theorem \ref{thm:sum_mult_bound}.
\begin{theorem}\label{thm:sum_mult_bound_q}
Let $q$ be a prime power. For all $m,n \in \mathbb{N}$, $h_q(mn) \geq h_q(m)+ h_q(n)$.
\end{theorem}
Theorem \ref{thm:mult2} generalises as follows.
\begin{theorem}
\label{thm:pboundforq}
Let $p$ be a prime, and let $q$ be a power of $p$. Then $h_q(pn)\leq ph_q(n)$. 
\end{theorem}

To prove Theorem \ref{thm:pboundforq}, we use the following lemma.
\begin{lemma}\label{lem:poly_sum}
Let $p$ be a prime, and let $q$ be a power of $p$. Suppose that $f$ is a polynomial of degree at most $p-1$ over $\F_q$. Then
\begin{equation}
    \sum_{x=0}^{p-1} f(x) = \begin{cases} 0 & \deg{f} < p-1 \\ -c_{p-1} & \deg{f} = p-1, \end{cases}
\end{equation}
where $c_{p-1}$ refers to the leading coefficient of $f$ when $f$ has degree $p-1$.
\end{lemma}

\begin{proof}
Lemma \ref{lem:poly_sum} will follow from the special case of monomials. If $f$ is constant, then clearly $\sum_x f(x)=pf(0) = 0$. If $f(x) = x^r$ for some $0 < r < p-1$, then let $\omega \in \Z/p\Z$ be a primitive root. Then,
\[
\sum_x x^r = \sum_x (\omega x)^r = \omega^r \sum_x x^r,
\]
where the first equality follows from the fact that multiplying by $\omega$ is a bijection on $(\Z/p\Z)^\times$. Thus $(\omega^r - 1)\sum_x f(x) = 0,$ and so $\sum_x f(x) = 0$ as $\omega$ is a primitive root. 
Finally, observe that $\sum_x x^{p-1}$ = $\sum_{x \neq 0} 1 = p-1$.
\end{proof}

\begin{proof}[Proof of Theorem \ref{thm:pboundforq}]
Given a vector $v \in \F_q^{pn}$, we will define its \emph{degree} as follows.
For $i\in\{0,\dots,n-1\}$ and $v \in \F_q^{pn}$, let 
\begin{align*}
\pi_i(v):\F_p&\to \F_q\\
j &\mapsto v_{nj+i}.
\end{align*}
Any function taking $\F_p$ to $\F_q$ can be written uniquely as the restriction to $\F_p$ of a polynomial of degree at most $p-1$ over $\F_q$. Let $d_i$ be the degree of $\pi_i(v)$. 
The degree $d$ of $v$ is defined as $d = \max_i d_i$.

Suppose that $V \subseteq \F_q^{pn}$ is the vector space spanned by $h_q(np)$ vectors that are linearly independent and work together. For $d\in\{0,\dots,p-1\}$, let $V_d \leq V$ be the subspace consisting of all vectors of degree at most $d$. Since every vector has degree at most $p-1$, we immediately get that $V_{p-1} = V$.

Then, Theorem \ref{thm:pboundforq} will follow from the following claim.
\begin{claim*}
If $V_{-1}$ is defined to be the trivial subspace, then for each $0 \leq d \leq p-1$, $\dim{V_d} \leq \dim{V_{d-1}} + h_q(n)$.
\end{claim*}

\begin{claimproof}[Proof of Claim]
Suppose that $v^{(1)}, \dots, v^{(r)}$ are vectors in $V_{d}$ that extend a basis of $V_{d-1}$ to a basis of $V_{d}$; thus, $r = \dim{V_d} - \dim{V_{d-1}}$. We must prove that $r \leq h_q(n)$.

For $t\in [r]$, define $w^{(t)} \in \F_q^n$ such that $w^{(t)}_i$ is the coefficient of $X^d$ in the polynomial corresponding to $\pi_i(v^{(t)})$. We claim that $w^{(1)},\dots,w^{(r)}$ are linearly independent and work together, proving $r\leq h_q(n)$.

We first show they work together. Let $x \in \F_q^n$. We must prove that there is a shift $k$ such that $w^{(t)} \cdot \sigma^k x = 0$ for each $t\in [r]$.

Define the vector $y \in \F_q^{pn}$ by $y_{ nj + i} = j^{p-1-d} x_i$, for $i\in\{0,\dots,n-1\}$ and $j\in \{0,\dots,p-1\}$. Since the $v^{(t)}$ work together, there must be a shift $k$ with the property that $v^{(t)} \cdot \sigma^k y = 0$ for all $t\in [r]$. By definition, 
\begin{align*}
    0=v^{(t)} \cdot \sigma^k y &= \sum_{i = 0}^{n-1}\sum_{j = 0}^{p-1} v^{(t)}_{nj+i} y_{nj + i - k}.
\end{align*}
Now, we will prove that
\begin{equation}\label{eqn:v_w_relation}
    \sum_{j = 0}^{p-1} v^{(t)}_{nj+i} y_{nj + i - k} = - w^{(t)}_i x_{i-k}
\end{equation}
for each $t \in [r]$, where the index of $x_{i-k}$ is taken modulo $n$.

To prove this, note that by definition of $w^{(t)}$, we may write 
\[
v^{(t)}_{nj+i} = (\pi^{(t)}_i(v))(j) = w^{(t)}_i j^d + f(j)
\]
where $f$ is some polynomial of degree at most $d-1$. There is a unique $\lambda \in \Z/p\Z$ with the property that, for each $j \in \{0, \dots, p-1\}$,
\[
nj + i - k = n(j + \lambda) + \ell,
\]
where $\ell \in \{0, \dots, n-1\}$ satisfies $\ell \equiv i-k \mod n$. This gives that $y_{nj + i - k} = x_{i - k}(j + \lambda)^{p-1-d}$, and thus
\begin{align*}
    \sum_{j = 0}^{p-1} v^{(t)}_{nj+i} y_{nj + i - k} &= \sum_{j=0}^{p-1} \left(w^{(t)}_i j^d + f(j)\right) x_{i - k}(j + \lambda)^{p-1-d} \\
    &= -w^{(t)}_ix_{i-k}
\end{align*}
by Lemma \ref{lem:poly_sum}.

It follows that
\begin{align*}
    0=-v^{(t)} \cdot \sigma^k y &= \sum_{i = 0}^{n-1} w^{(t)}_ix_{i-k} = w^{(t)} \cdot \sigma^k x,
\end{align*}
for each $t\in [r]$. This proves that the $w^{(t)}$ work together.

It remains to show that the $w^{(t)}$ are linearly independent. Suppose otherwise; then, there must exist a linear relation
\[
\sum_{t=1}^r \lambda_t w^{(t)} = 0
\]
for some $\lambda_t \in \F_q$ not all zero.

By definition, this means that each $\pi_i(\sum_{t=1}^r \lambda_t v^{(t)}) : \F_p \rightarrow \F_q$ has the coefficient of $X^d$ equal to zero, and thus has degree at most $d-1$. In other words, $\sum_{t=1}^r \lambda_t v^{(t)} \in V_{d-1}$, contradicting the assertion that the $v^{(t)}$ extend a basis of $V_{d-1}$.
\end{claimproof}

Applying the claim, we conclude that
\begin{align*}
h_q(pn)=\dim V &= \sum_{d=0}^{p-1} \left(\dim{V_{d}} - \dim{V_{d-1}}\right) \leq p h_q(n). \qedhere
\end{align*}
\end{proof}

Since the cyclic shifts of a subspace $U \leq \mathbb{F}_q^n$ can only cover at most $|U| \cdot n$ vectors in $\mathbb{F}_q^n$, if $U$ is cyclically covering, we must have $|U| \geq q^n/n$. This implies $\dim(U) = \log_q(|U|) \geq n - \log_q(n)$, and $h_q(n) \leq \log_q(n)$. In particular, for $n<q$ we have $h_q(n) = 0$, and this gives the following corollary of Theorem \ref{thm:pboundforq}.
\begin{corollary}
Let $p$ be prime, and let $q$ be a power of $p$. Then $h_q(\ell p^d)=0$ for any $\ell <q$. 
\end{corollary}
\begin{remark*}
In Theorem \ref{thm:pboundforq}, we require that the characteristic of the field $\F_q$ is the same $p$ as in the bound. For example, $h_3(2n) \leq 2h_3(n)$ fails for $n = 4$. Indeed, observe that $h_3(4) = 0$ (as noted in \cite{CameronEllisRaynaud18}), whereas $h_3(8) = 1$ follows from \cite[Theorem 5]{CameronEllisRaynaud18}.
\end{remark*}

For $q>2$, we have the following stronger version of Theorem \ref{thm:main_theorem}.
\begin{theorem}\label{thm:main_thm_hq}
Suppose that $q$ is an odd prime, and $p > q$ is a prime with $q$ as a primitive root. Then $h_q(p) = 0$; in other words, there are \emph{no} nonzero vectors $v \in \F_q^p$ that work.
\end{theorem}

\begin{remark*}
It is known by a result of Heath-Brown \cite{HeathBrown86} that for all but at most two primes $q$, there are infinitely many primes $p$ such that $q$ is a primitive root of $p$. In particular, we know unconditionally that for all but two primes $q$, $h_q(n)$ does not tend to infinity among $n$ coprime to $q$. It is widely believed that there are no primes $q$ that are primitive roots for only finitely many $p$; if this were true, we would know that $h_q(n)$ does not tend to infinity for any $q$ (among $n$ coprime to $q$).
\end{remark*}

\begin{proof}
Suppose that $v \in \F_q^p$ is nonzero. We consider the polynomial $f_v \in \F_q[X] /(X^p - 1)$ which corresponds to a vector $v$, after making the natural modifications to the definition given in Section \ref{sec:rootsandconj}. An analogue of Proposition \ref{prop:think_about_polys} holds over $\F_q$, with the same proof. Thus, $v$ works if and only if, for every $x \in \F_q^p$, there is a $k$ such that the coefficient of $X^k$ in $f_v f_x$ is $0$. That is  $v$ fails to work if and only if there exists $x$ such that $f_v f_x$ has no zero coefficients. In other words, $v$ fails to work if and only if $f_v$ is a factor of a polynomial with no zero coefficients.

Since $q$ is a primitive root for the prime $p$, the analogue of Lemma \ref{lem:factors_of_cyclo} tells us that $X^p - 1$ factors into irreducible polynomials as 
\[
X^p - 1 = (X - 1)(1 + X + \dots + X^{p-1}).
\]

Then we can write  \[\frac{\F_q[X]}{( X^p - 1 )} =  \frac{\F_q[X]}{( X - 1 )} \oplus \frac{\F_q[X]}{( 1 + X + X^2 + \dotsb + X^{p-1} )}.\] 
Suppose $f_v$ is non-zero in $\F_q[X]/( X^p - 1 )$. Write $a \in \F_q[X]/(X-1)$ for $f_v \mod X-1$ and $b \in \F_q[X]/(1 + X + \dots + X^{p-1})$ for $f_v \mod 1 + X + \dots + X^{p-1}$. The assertion that $f_v$ is nonzero is then equivalent to the assertion that either $a$ or $b$ is nonzero.

Suppose first that $a$ is nonzero. Since $\F_q[X]/(X - 1)$ is a field, there is an inverse $a^{-1}$ for $a$. Let $c = 1 + X + \dots + X^{p-1} \mod X-1$, and let $g \in \F_q[X]/(X^p - 1)$ denote the polynomial that is $ca^{-1} \mod X-1$ and $0 \mod 1 + X + \dots + X^{p-1}$. Thus, $f_vg$ is $c \mod X-1$ and $0 \mod 1 + X + \dots + X^{p-1}$, and so is equal to $1 + X + \dots + X^{p-1} \mod X^p - 1$. In particular, $f_v$ is a factor of $1 + X + \dots + X^{p-1}$.

On the other hand, suppose that $a$ is zero. Then $b$ must be nonzero, and a similar argument shows that $f_v$ is a factor of $X - 1$.
Thus, any nonzero polynomial in $\F_q[X]/( X^p - 1 )$ is a factor of either $X - 1$ or $1 + \dots + X^{p-1}$.

Now, $1 + \dots + X^{p-1}$ itself has no coefficient equal to zero. Also, $X-1$ is a factor of
\[(X-1) (X^{p-2} + X^{p-4} + \dotsb + X - 1) = 1 - 2X + X^2 + \sum_{i=3}^{p-1} (-X)^i,\]
which also has no coefficient equal to zero.

Thus, any nonzero polynomial is a factor of a polynomial with no coefficient equal to zero, and so any nonzero vector does not work.
\end{proof}

\section{Conclusion}
\label{sec:concl}
Let $\ord_p(2)$ denote the order of $2$ in $\left(\Z/p\Z\right)^{\times}$.
We have shown that $h_2(p)\leq 2$ for all Artin primes $p$, where Artin primes are exactly the primes satisfying $\ord_p(2)=p-1$. It would be interesting to see whether $h_2(p)$ is still small if $p$ is ``almost'' an Artin prime. For example, is there a function $f:\N \to \N$ such that $h_2(p)\leq f\left(\frac{p-1}{\ord_p(2)}\right)$?

Since we now know $2$ appears infinitely often in the multiset $\{h_2(n) : n\in \N\}$ (assuming Artin's conjecture is true), another interesting direction for future work is to see which other numbers appear infinitely often. 
\begin{problem}
For which $k\in \N$ are there infinitely many $n$ such that $h_2(n)=k$?
\end{problem}
The lower bound in Theorem \ref{thm:sum_mult_bound} suggests that one might at least expect there to be infinitely many such $k$. Indeed, a bound on $h_2(3n)$ which is only a function of $h_2(n)$ would suffice for this. To this end we propose the following problem.
\begin{problem}
For which $k \in \N$ is there a function $f_k : \mathbb{N} \to \mathbb{N}$ such that $h_2(kn)\leq f_k(h_2(n))$ for all $n\in \N$?
\end{problem}

From our results, it follows that $h_2(a_n)\to \infty$ whenever the number of odd prime factors of $a_n$ tends to infinity. Our computer searches show that the sequence $h_2(3 \cdot 2^n)$ begins 1, 2, 3, 3, but we are unable to determine for which $n$ we have $h_2(3 \cdot 2^{n}) < h_2(3 \cdot 2^{n+1})$. It would be interesting to determine for which $x$ we have $h_2(2x)>h_2(x)$. We are not even able to determine if $h_2(3\cdot 2^n)\to \infty$ as $n \to \infty$, and more generally, we pose the following question.
\begin{problem}
For which $k\in \N$ do we have $h_2(k \cdot 2^n)\to \infty$?
\end{problem}

We remark that we also leave open our own conjecture that the only vectors that work with $\e$ are the symmetric vectors (Conjecture \ref{conj:our_in_intro}).
We have verified that Conjecture \ref{conj:our_in_intro} holds for all odd $n$ with $3 \leq n \leq 43$, if $7 \nmid n$.
In the case that $n = 7$, there are precisely 12 non-symmetric vectors $v$ that work with $\e$. The first of these is the example $v = (0,1,1,0,0,0,0)$ given in Observation \ref{obs:counterexample_7}, and the other 11 can derived from this by scaling by non-zero $\ell$ (in other words, replacing $x_i$ by $x_{\ell i}$ for every $i$), replacing $v$ with $v + \e$, and combinations of the two. For $n = 21$ and $n = 35$, the only non-symmetric vectors $v$ that work with $\e$ may be constructed from the $n=7$ case in a systematic way.

Finally, the following problem posed by Cameron, Ellis and Raynaud \cite{CameronEllisRaynaud18}, to which we have shown some partial progress, remains very interesting.
\begin{problem}
For which $n$ is $h_q(n)=0$?
\end{problem}

\paragraph{Acknowledgements}
We would like to thank the referees for their valuable feedback which vastly improved the readability of this paper.

\bibliographystyle{plain}
\bibliography{cyclicallycovering}
\appendix

\section{Known bounds on \texorpdfstring{$h_2(n)$}{h2(n)} for small \texorpdfstring{$n$}{n}}
\label{sec:comp_exp}
The following two theorems of Cameron, Ellis and Raynaud provide several exact values of $h_2(n)$.
\begin{theorem}[{\cite[Theorem 5]{CameronEllisRaynaud18}}]
\label{thm:CERqdminus1}
If $q$ is a prime power and $d \in \N$, then
$h_q(q^d-1) = d-1$.
\end{theorem}
\begin{theorem}[{\cite[Theorem 8]{CameronEllisRaynaud18}}]
\label{thm:CERsum}
Let $q$ be a prime power, and let $k, d \in \N$ such that $\gcd(d+ 1, q^k-1) = 1$. Set $n = \sum_{r=0}^d q^{kr} = \frac{q^{k(d+1)}-1}{q^k-1}$. 
Then $h_q(n) = kd$.
\end{theorem}
In addition, we make use of the following bounds.
\begin{theorem}[{\cite[Lemma 4]{CameronEllisRaynaud18}}]
\label{thm:CERlog}
For $q$ a prime power and $n \in  \mathbb{N}$, we have $h_q(n) \leq \floor{\log_q(n)}$.
\end{theorem}
\begin{theorem}[{\cite[Lemma 2]{CameronEllisRaynaud18}}]
\label{thm:CERlower2}
For odd positive integers $n > 3$, we have $h_2(n) \geq 2$.
\end{theorem}
Combining these results with our results and brute force calculations, we obtain the lower bounds $\ell$ and upper bounds $u$ on $h_2(n)$ for small values of $n$ given in Table \ref{tab:1}.

\begin{table}[ht]
\renewcommand{\arraystretch}{1.4}
\centering
\begin{tabular}{|c|c|c|c|}
\hline
$n$ & $(\ell, u)$& Reason for $\ell$ & Reason for $u$\\\hline
  1 & (0,0) & & \\
  2 & (0,0) &                                  &         Theorem \ref{thm:mult2} \\
  3 & (1,1) &                       Theorem \ref{thm:CERsum} &        Theorem \ref{thm:CERlog}\\
  4 & (0,0) &                                 &          Theorem \ref{thm:mult2}\\
  5 & (2,2) &                       Theorem \ref{thm:CERsum}&   Theorem \ref{thm:main_theorem}\\
  6 & (2,2) &    Brute force  &         Theorem \ref{thm:mult2}\\
  7 & (2,2) &                       Theorem \ref{thm:CERsum} &        Theorem \ref{thm:CERlog}\\
  8 & (0,0) &                                   &        Theorem \ref{thm:mult2}\\
  9 & (3,3) &                       Theorem \ref{thm:CERsum} &        Theorem \ref{thm:CERlog}\\
 10 & (2,2) &    Theorem \ref{thm:sum_mult_bound} &       Brute force\\
 11 & (2,2) &                    Theorem \ref{thm:CERlower2} &  Theorem \ref{thm:main_theorem}\\
 12 & (3,3) &                      Brute force &        Theorem \ref{thm:CERlog}\\
 13 & (2,2) &                    Theorem \ref{thm:CERlower2} &  Theorem \ref{thm:main_theorem} \\
 14 & (3,3) &                      Brute force &        Theorem \ref{thm:CERlog}\\
 15 & (3,3) &                       Theorem \ref{thm:CERsum} &        Theorem \ref{thm:CERlog}\\
 16 & (0,0) &                                  &         Theorem \ref{thm:mult2}\\
 17 & (4,4) &                       Theorem \ref{thm:CERsum} &        Theorem \ref{thm:CERlog}\\
 18 & (3,3) &    Theorem \ref{thm:sum_mult_bound} &       Brute force\\
 19 & (2,2) &                    Theorem \ref{thm:CERlower2} &  Theorem \ref{thm:main_theorem}\\
 20 & (3,3) &                      Brute force &       Brute force\\
 21 & (3,4) &    Theorem \ref{thm:sum_mult_bound} &        Theorem \ref{thm:CERlog}\\
 22 & (2,4) &   Theorem \ref{thm:sum_mult_bound}  &         Theorem \ref{thm:mult2}\\
 23 & (3,3) &                      Brute force &       Brute force\\
 24 & (3,3) &   Theorem \ref{thm:sum_mult_bound}  &        Brute force\\
 25 & (4,4) &    Theorem \ref{thm:sum_mult_bound} &        Theorem \ref{thm:CERlog}\\
 26 & (2,4) &   Theorem \ref{thm:sum_mult_bound}  &         Theorem \ref{thm:mult2}\\
 27 & (4,4) &    Theorem \ref{thm:sum_mult_bound} &        Theorem \ref{thm:CERlog}\\
 28 & (3,4) &   Theorem \ref{thm:sum_mult_bound}  &        Theorem \ref{thm:CERlog}\\
 29 & (2,2) &                    Theorem \ref{thm:CERlower2} &  Theorem \ref{thm:main_theorem}\\
 \hline
 \end{tabular}
 \caption{The upper and lower bounds on $h_2(n)$ we are aware of, along with their respective sources, are given for $n\in \{1,\dots,29\}$.}
   \label{tab:1}
\end{table}

\end{document}